\nonstopmode
\documentclass[12pt,reqno]{amsart}
\usepackage{graphicx}
\usepackage{latexsym}
\usepackage{fancyhdr}
\usepackage{amsmath, amssymb}
\usepackage[all]{xy}
\newdir{ >}{{}*!/-10pt/@{>}}
\newdir^{ (}{{}*!/-5pt/@^{(}}
\newdir_{ (}{{}*!/-5pt/@_{(}}
\usepackage{pdflscape}
\usepackage{longtable}
\usepackage{rotating}
\usepackage{verbatim}
\usepackage{hyperref}
\usepackage{subfigure}
\usepackage{mathrsfs}
\usepackage{tensor}
\usepackage{mdwlist}
\usepackage{etoolbox}
\usepackage{mathtools}
\usepackage{todonotes}

\theoremstyle{plain}
\newtheorem*{lemma*}{Lemma}
\newtheorem{lemma}{Lemma}
\newtheorem*{theorem*}{Theorem}
\newtheorem{theorem}{Theorem}
\newtheorem*{proposition*}{Proposition}
\newtheorem{proposition}{Proposition}
\newtheorem*{corollary*}{Corollary}
\newtheorem{corollary}{Corollary}
\newtheorem*{claim*}{Claim}
\theoremstyle{definition}
\newtheorem*{assumption*}{Assumption}

\newtheorem*{definition*}{Definition}

\newtheorem*{convention*}{Convention}
\newtheorem*{example*}{Example}

\newtheorem*{algorithm*}{Algorithm}
\newtheorem*{remark*}{Remark}
\newtheorem{remark}{Remark}

\newtheorem*{remarks*}{Remarks}
\newtheorem*{openproblem*}{Open Problem}
\newtheorem{openproblem}{Open Problem}

\numberwithin{equation}{section}

\def\al{\alpha}
\def\be{\beta}
\def\ga{\gamma}
\def\de{\delta}
\def\ep{\epsilon}

\def\th{\theta}

\def\la{\lambda}

\def\rh{\rho}

\def\si{\sigma}

\def\ta{\tau}

\def\ph{\phi}
\def\vh{\varphi}

\def\ps{\psi}

\def\Ga{\Gamma}
\def\De{\Delta}

\def\La{\Lambda}

\def\Ph{\Phi}

\def\Om{\Omega}

\def\C{\mathbb{C}}

\def\N{\mathbb{N}}

\def\R{\mathbb{R}}
\def\Z{\mathbb{Z}}

\def\cA{\mathcal{A}}
\def\cB{\mathcal{B}}
\def\cC{\mathcal{C}}

\def\cF{\mathcal{F}}

\def\cI{\mathcal{I}}
\def\cJ{\mathcal{J}}
\def\cK{\mathcal{K}}
\def\cL{\mathcal{L}}

\def\p{\partial}
\def\oI{\overline I}
\def\ol{\overline}
\def\ul{\underline}

\def\Hoeld{\on{H\ddot{o}ld}}
\def\Lip{\on{Lip}}

\def\const{M}

\renewcommand{\Re}{\mathrm{Re}}
\renewcommand{\Im}{\mathrm{Im}}
\def\<{\langle}
\def\>{\rangle}
\renewcommand{\o}{\circ}
\def\cq{{/\!\!/}}

\let\on=\operatorname
\newcommand{\sr}[1]%
{\ifmmode{}^\dagger\else${}^\dagger$\fi\ifvmode
\vbox to 0pt{\vss
 \hbox to 0pt{\hskip\hsize\hskip1em
 \vbox{\hsize3cm\raggedright\pretolerance10000
 \noindent #1\hfill}\hss}\vss}\else
 \vadjust{\vbox to0pt{\vss%
 \hbox to 0pt{\hskip\hsize\hskip1em%
 \vbox{\hsize3cm\raggedright\pretolerance10000%
 \noindent #1\hfill}\hss}\vss}}\fi%
}

\title[Optimal Sobolev regularity of roots of polynomials]
{Optimal Sobolev regularity of roots of polynomials}

\author[Adam Parusi\'nski and  Armin Rainer]
{Adam Parusi\'nski and Armin Rainer}

\address {Adam Parusi\'nski: Univ. Nice Sophia Antipolis, CNRS,  LJAD, UMR 7351, 06108 Nice, France}

\email{adam.parusinski@unice.fr}

\address{Armin Rainer: Fakult\"at f\"ur Mathematik, Universit\"at Wien, 
Oskar-Morgenstern-Platz~1, A-1090 Wien, Austria}

\email{armin.rainer@univie.ac.at}

\begin{document}

\begin{abstract}
  We study the regularity of the roots of complex univariate polynomials 
  whose coefficients depend smoothly on parameters. We show that any 
  continuous choice of the roots of a $C^{n-1,1}$-curve of monic polynomials 
  of degree $n$ is locally absolutely continuous with locally $p$-integrable 
  derivatives for every $1 \le p < n/(n-1)$, uniformly with respect to the coefficients. 
  This result is optimal: in general, the derivatives of the roots of a smooth 
  curve of monic polynomials of degree $n$ are not locally $n/(n-1)$-integrable, 
  and the roots may have locally unbounded variation if the coefficients are only 
  of class $C^{n-1,\alpha}$ for $\alpha <1$. 
  We also prove a generalization of Ghisi and Gobbino's higher order Glaeser inequalities. 
  We give three applications of the main results: local solvability of a system of 
  pseudo-differential equations, a lifting theorem for mappings into orbit spaces of finite 
  group representations, and a sufficient condition for multi-valued functions to be 
  of Sobolev class $W^{1,p}$ in the sense of Almgren.  
\end{abstract}

\thanks{Supported by the Austrian Science Fund (FWF), Grant P~26735-N25, and by ANR project STAAVF (ANR-2011 BS01 009).}
\date{\today}
 
\maketitle

\tableofcontents

\section{Introduction}

This paper is dedicated to the problem of determining the optimal regularity of the roots of 
univariate polynomials whose coefficients depend smoothly on parameters. There is a vast literature on this 
problem, but most contributions treat special cases: 
\begin{itemize}
    \item the polynomial is assumed to have only real roots 
    (\cite{Bronshtein79}, \cite{Mandai85}, \cite{Wakabayashi86}, \cite{AKLM98}, \cite{KLM04}, 
    \cite{BBCP06}, \cite{BonyColombiniPernazza06}, \cite{Tarama06}, \cite{BonyColombiniPernazza10}, 
    \cite{ColombiniOrruPernazza12}, 
    \cite{ParusinskiRainerHyp}),
    \item only radicals of functions are considered (\cite{Glaeser63R}, \cite{CJS83}, \cite{Tarama00}, \cite{CL03}, 
    \cite{GhisiGobbino13}),    
    \item it is assumed that the roots meet only of finite order, e.g., 
    if the coefficients are real analytic or in some other quasianalytic class,
    (\cite{CC04}, \cite{RainerAC}, \cite{RainerQA}, \cite{RainerOmin}, \cite{RainerFin}),
    \item quadratic and cubic polynomials (\cite{Spagnolo99}), etc. 
\end{itemize}  
In this paper we consider the general case: let $(\al,\be) \subseteq \R$ be a bounded open interval and let 
  \begin{equation} \label{curveofpolynomials}
    P_a(t)(Z)= P_{a(t)}(Z) = Z^n + \sum_{j=1}^n a_j(t) Z^{n-j}, \quad t \in (\al,\be), 
  \end{equation}
be a monic polynomial whose coefficients are complex valued smooth functions $a_j : (\al,\be) \to \C$, $j = 1,\ldots,n$. 
It is not hard to see that $P_a$ always admits a continuous system of roots (e.g.\ \cite[Ch.~II Theorem~5.2]{Kato76}), 
but in general the roots cannot satisfy a local Lipschitz condition. 
For a long time it was unclear whether the roots of $P_a$ admit locally absolutely continuous parameterizations. 
This question was 
affirmatively solved in our recent paper \cite{ParusinskiRainerAC}: there is a positive integer $k=k(n)$ and a 
rational number $p=p(n)>1$ such that, if the coefficients are of class $C^k$, 
then each continuous root $\la$ is locally absolutely continuous 
with derivative $\la'$ being locally $q$-integrable for each $1 \le q < p$, uniformly with respect to the coefficients. 

The problem of absolute continuity of the roots arose in the analysis of certain systems of pseudo-differential 
equations due to Spagnolo \cite{Spagnolo00}; see Section \ref{PDE}.  
For the history of the problem we refer to the introduction of \cite{ParusinskiRainerAC}. 
The main tool of \cite{ParusinskiRainerAC} was the resolution of singularities. With this technique we could not determine the 
optimal parameters $k$ and $p$.

\subsection{Main results}

In the present paper we prove the optimal result by elementary methods.
Our main result is the following theorem.

\begin{theorem} \label{main}
  Let $(\al,\be) \subseteq \R$ be a bounded open interval and let $P_a$   
  be a monic polynomial \eqref{curveofpolynomials} with coefficients $a_j \in C^{n-1,1}([\al,\be])$, $j = 1,\ldots,n$. 
  Let $\la \in C^0((\al,\be))$ be a continuous root of $P_a$ on $(\al,\be)$.
  Then $\la$ is absolutely continuous on $(\al,\be)$ and belongs to the Sobolev space 
  $W^{1,p}((\al,\be))$ for every $1 \le p < n/(n-1)$. The derivative $\la'$ satisfies  
  \begin{align} \label{bound} 
   \| \la' \|_{L^p((\al,\be))}  
   &\le C(n,p) \max\{1, (\be-\al)^{1/p}\} 
   \max_{1 \le j \le n} \|a_j\|^{1/j}_{C^{n-1,1}([\al,\be])},
  \end{align}
  where the constant $C(n,p)$ depends only on $n$ and $p$.
\end{theorem}

A well-known estimate for the Cauchy bound of a polynomial (cf.\ \cite[p.56]{Malgrange67} or \cite[{(8.1.11)}]{RS02}) 
gives $|\la(t)| \le 2 \max_{1 \le j \le n} |a_j(t)|^{1/j}$ for all $t \in (\al,\be)$, and hence 
\begin{equation*}
  \| \la \|_{L^p((\al,\be))} \le C(n)  (\be-\al)^{1/p} \max_{1 \le j \le n} \|a_j\|_{L^\infty((\al,\be))}^{1/j}.
\end{equation*}
It follows that 
\begin{align} \label{Wbound} 
   \| \la \|_{W^{1,p}((\al,\be))}  
   &\le C(n,p) \max\{1, (\be-\al)^{1/p}\} 
   \max_{1 \le j \le n} \|a_j\|^{1/j}_{C^{n-1,1}([\al,\be])},
  \end{align}
An application of H\"older's inequality yields the following corollary.

\begin{corollary}
  Every continuous root of $P_a$ on $(\al,\be)$ is H\"older continuous of exponent $\ga = 1-1/p < 1/n$, 
  and 
  \begin{equation} \label{Cbound}
     \| \la \|_{C^{0,\ga}([\al,\be])}  
   \le  C(n,p) \max\{1, (\be-\al)^{1/p}\} 
   \max_{1 \le j \le n} \|a_j\|^{1/j}_{C^{n-1,1}([\al,\be])}.
  \end{equation}
\end{corollary}

\begin{proof} 
	Indeed, $|\la(t) - \la(s)| \le |\int_s^t \la' \,d\ta| \le \|\la'\|_{L^p((\al,\be))} |t-s|^{1-1/p}$.
\end{proof}

The result in Theorem \ref{main} is best possible in the following sense:
\begin{itemize}
  \item In general the roots of a polynomial of degree $n$ cannot lie locally in $W^{1,n/(n-1)}$, 
  even when the coefficients are real analytic. 
  For instance, $Z^n = t$, $t \in \R$. 
  \item If the coefficients are just in $C^{n-1,\de}([\al,\be])$ for every $\de<1$, 
  then the roots need not have bounded variation in 
  $(\al,\be)$. See \cite[Example 4.4]{GhisiGobbino13}. 
\end{itemize}

A \emph{curve} of complex monic polynomials \eqref{curveofpolynomials} admits a continuous choice of its roots. This is 
no longer true if the dimension of the parameter space is at least two. In that case monodromy may 
prevent the existence of continuous roots. However, we obtain the following multiparameter result, where we 
impose the existence of a continuous root; see also Remark~\ref{remark4}. 

\begin{theorem}  \label{main2}
  Let $U \subseteq \R^m$ be open and let  
  \begin{equation} \label{polynomialmulti}
    P_a(x)(Z)= P_{a(x)}(Z) = Z^n + \sum_{j=1}^n a_j(x) Z^{n-j}, \quad x \in U,
  \end{equation}
  be a monic polynomial with coefficients $a_j \in C^{n-1,1}(U)$, $j = 1,\ldots,n$.   
  Let $\la \in C^0(V)$ be a root of $P_a$ on a relatively compact open subset $V \Subset U$. 
  Then $\la$ belongs to the Sobolev space $W^{1,p}(V)$ for every $1 \le p < n/(n-1)$. 
  The distributional gradient $\nabla \la$ satisfies  
  \begin{equation} \label{multbound} 
   \|\nabla \la \|_{L^p(V)}  \le  C(m,n,p,\cK ) \max_{1 \le j \le n} \|a_j\|^{1/j}_{C^{n-1,1}(\overline W)},
  \end{equation}
  where $\cK$ is any finite cover of $\overline V$ by open boxes $\prod_{i=1}^m (\al_i,\be_i)$ contained 
  in $U$ and $W = \bigcup \cK$; 
  the constant $C(m,n,p,\cK )$ depends only on $m$, $n$, $p$, and the cover $\cK$.   
\end{theorem}

\begin{remark}
	For any two distinct points $x$ and $y$ in $V$ such that the segment $[x,y]$ is contained in $V$, the root $\la$ satisfies a 
	H\"older condition 
	\[
		\frac{|\la(x) - \la(y)|}{|x-y|^\ga} 
		\le C(m,n,p,\on{diam}(V))  
   		\max_{1 \le j \le n} \|a_j\|^{1/j}_{C^{n-1,1}([x,y])}, 
	\]
	where $\ga = 1-1/p < 1/n$. This follows easily from Theorem \ref{main2} and Remark \ref{remark4}. 
\end{remark}

The proof of Theorem \ref{main} makes essential use of the recent result of Ghisi and Gobbino \cite{GhisiGobbino13} 
who found the optimal regularity of radicals of functions (we will need a version for complex valued functions; 
see Section \ref{radicals}). But we independently prove and generalize Ghisi and Gobbino's higher order 
Glaeser inequalities (see Section \ref{Glaeser}) 
on which their 
result is based.

\begin{theorem}[Ghisi and Gobbino \cite{GhisiGobbino13}] \label{GhisiGobbino}
  Let $k$ be a positive integer, let $\al \in (0,1]$, let $I \subseteq \R$ be an open bounded interval, and let $f : I \to \R$ 
  be a function.
  Assume that $f$ is continuous and that there exists $g \in C^{k,\al}(\oI,\R)$ such that 
  \[
    |f|^{k +\al} = |g|.
  \]
  Let $p$ be defined by $1/p + 1/(k+\al) =1$.
  Then we have $f' \in L^p_w(I)$ and 
  \begin{equation} \label{GG}
    \|f'\|_{p,w,I} \le 
    C(k) \max\Big\{\big(\Hoeld_{\al,I}(g^{(k)})\big)^{1/(k+\al)}|I|^{1/p}, 
    \|g'\|_{L^\infty(I)}^{1/(k+\al)}\Big\}, 
  \end{equation}
  where $C(k)$ is a constant that depends only on $k$.
\end{theorem}

Here $L^p_w(I)$ denotes the weak Lebesgue space equipped with the quasinorm $\|\cdot\|_{p,w,I}$ (see Section \ref{Lebesgue}), 
and $\Hoeld_{\al,I}(g^{(k)})$ is the $\al$-H\"older constant of $g^{(k)}$ on $I$. 

\subsection{Open problems}

We remark that our bound \eqref{bound} is not invariant under rescaling, 
in contrast to \eqref{GG}. The reasons for this defect is linked to our method of proof. 

\begin{openproblem}
  Are there scale invariant estimates which could replace \eqref{bound}?    
\end{openproblem}

We do not know whether, in the setting of Theorem \ref{main}, $\la'$ is actually an element of $L^{n/(n-1)}_w((\al,\be))$; 
as one could expect in view of Theorem \ref{GhisiGobbino}. 
This has technical reasons and comes from the fact that $\|\cdot\|_{p,w,I}^p$ is not $\si$-additive.

\begin{openproblem}
  Is $\la'$ in the setting of Theorem \ref{main} an element of $L^{n/(n-1)}_w((\al,\be))$?
  If so is there an explicit bound for $\|\la'\|_{n/(n-1),w,(\al,\be)}$ in terms of the coefficients $a_j$ and the 
  interval $(\al,\be)$?  
\end{openproblem}

The roots of \eqref{polynomialmulti} will in general not allow for continuous 
(and, a fortiori, $W^{1,1}_{\on{loc}}$) parameterizations if $m\ge 2$. It is thus natural to ask if the roots are 
representable locally  
by functions of bounded variation.

\begin{openproblem}
  Are the roots of a polynomial $P_a(x)$, $x \in \R^m$, $m\ge 2$, 
  with smooth complex valued coefficients representable by functions which locally have bounded variation?
  We can prove this for radicals of smooth functions.     
\end{openproblem}

\subsection{Strategy of the proof of Theorem \ref{main}}

Let us briefly describe the strategy of our proof of Theorem \ref{main}. 
It is by induction on the degree of the polynomial and its heart is Proposition~\ref{induction} below.

First we reduce the polynomial $P_a$ to Tschirnhausen form $P_{\tilde a}$ (indicated by adding \emph{tilde}), 
where $\tilde a_1 \equiv 0$ 
(see Section \ref{Tschirnhausen}). 
This has the benefit that near points $t_0$, where not all coefficients vanish,
the polynomial $P_{\tilde a}$ splits, 
\[
  P_{\tilde a}(t) = P_b(t) P_{b^*}(t), \quad t \in I, \quad (t_0 \in I),
\]
thanks to the inverse function theorem.
It is important for our proof that the splitting is \emph{universal} (and independent of $t_0$). We achieve this by 
considering the polynomial
\begin{gather*}
  Q_{\underline a}(Z) := \tilde a_k^{- n/k} P_{\tilde a} (\tilde a_k^{1/k} Z) 
  = Z^n + \sum_{j=2}^n \tilde a_k^{- j/k} \tilde a_j Z^{n-j}, \quad \tilde a_k \ne 0,  
\end{gather*}
which splits locally near every $(\underline a_2,\ldots,\underline a_n) \in \C^{n-1}$, since $\underline a_k =1$. 
We obtain a universal splitting by choosing a finite subcover of the compact set of points with $\underline a_k =1$ and 
$|\underline a_j| \le 1$ for $j \ne k$. 
It induces a splitting of $P_{\tilde a}$ and 
gives formulas for the coefficients $b_i$ (and $b^*_i$) in terms of $\tilde a_j$. See Sections \ref{ssec:split} and 
\ref{universal}.
The differentiability class of the $\tilde a_j$ is preserved by the splitting.

After the Tschirnhausen transformation $P_b \leadsto P_{\tilde b}$, we split $P_{\tilde b}$ near points $t_1 \in I$, 
where not all $\tilde b_i$ vanish, 
\[
  P_{\tilde b}(t) = P_c(t) P_{c^*}(t), \quad t \in J, \quad (t_1 \in J).
\]
Again we use the universal splitting (now for degree $n_b := \deg P_b$ polynomials in Tschirnhausen form). 
We get formulas for $c_h$ (and $c^*_h$) in terms of $\tilde b_j$, 
and the differentiability class is preserved.
Apply the Tschirnhausen transformation $P_c \leadsto P_{\tilde c}$.

The central idea underlying the induction is to show that, for $1 \le p < n/(n-1)$, we have an estimate of the form 
\begin{align} \label{central}
  \||J|^{-1} |\tilde b_{\ell}(t_1)|^{1/\ell} \|_{L^p (J)}
  + \sum_{h=2}^{n_{c}} \|(\tilde c_{h}^{1/h})'\|_{L^p (J)}
  &\le C \Big( \| |I|^{-1}  {|\tilde a_k(t_0)|^{1/k}} \|_{L^p (J)} 
  + \sum_{i=2}^{n_b} \|(\tilde b_i^{1/i})'\|_{L^p (J)}\Big), 
\end{align}
for a universal constant $C=C(n,p)$ (where $n_c := \deg P_c$). 
Here $k$ (resp.\ $\ell$) is chosen such that $|\tilde a_k(t_0)|^{1/k} = \max_{2 \le j \le n} |\tilde a_j(t_0)|^{1/j}$ 
(resp.\ $|\tilde b_\ell(t_1)|^{1/\ell} = \max_{2 \le i \le n_b} |\tilde b_i(t_1)|^{1/i}$), that is 
the $k$th (resp.\ $\ell$th) correctly weighted coefficient is \emph{dominant} at $t_0$ (resp.\ $t_1$).

In the derivation of \eqref{central} we make essential use of \eqref{GG} and Lemma \ref{taylor} below in 
order to bound the left-hand side 
by 
\[
  |J|^{-1+1/p} |\tilde b_{\ell}(t_1)|^{1/\ell}.
\]
Now the key to get \eqref{central} from this is that we can choose the interval $J$ such that 
\begin{equation} \label{introkey}
   D |J|^{-1+1/p} |\tilde b_\ell(t_1)|^{1/\ell} = 
  | J|^{1/p} \Big( |I|^{-1}  {|\tilde a_k(t_0)|^{1/k}}  
  + \sum_{i=2}^{n_b} \|(\tilde b _i^{1/i})'\|_{L^1 (J)} \Big) 
\end{equation}
where $D$ is a universal constant.

We get the estimate \eqref{central} on neighborhoods $J$ of all points $t_1 \in I$, where not all $\tilde b_i$ vanish. 
In order to glue these estimates 
we prove in Proposition \ref{cover} that there is a countable subcollection of intervals $J$ such that 
every point in their union is covered at most by two intervals. 
In this gluing process we use the $\si$-additivity of $\|\cdot\|^p_{L^p}$. 
Since the $L^p_w$-quasinorm lacks this property, we are forced to switch from $L^{n/(n-1)}_w$- to $L^p$-bounds for $p<n/(n-1)$.

In the end we must estimate the right-hand side of \eqref{central} by a bound involving the $C^{n-1,1}$-norm of the 
$\tilde a_j$. At this stage we will not always have an identity corresponding to \eqref{introkey} 
(see Remark \ref{rem:blowup}). 
We resolve this inconvenience by extending the coefficients $\tilde a_j$ 
to a larger interval and we force them to vanish at the boundary 
of this interval. This results in an identity of the type \eqref{introkey} for the $\tilde a_j$ instead of the $\tilde b_i$
(see Lemma \ref{lem:whitney}). 
However, in this process we lose scale invariance of our bound \eqref{bound}.

\subsection{Structure of the paper}

The paper is structured as follows. We fix notation and recall facts on function spaces in Section \ref{functionspaces}. 
Ghisi and Gobbino's result on radicals (Theorem \ref{GhisiGobbino}) is extended to complex valued functions 
in Section \ref{radicals}.
We collect preliminaries on polynomials and define a universal splitting of such in Section \ref{polynomials}.
We derive bounds for the coefficients of a polynomial and generalize Ghisi and Gobbino's higher order Glaeser inequalities 
\cite[Proposition~3.4]{GhisiGobbino13}
in Section~\ref{Glaeser}, by applying these bounds to the Taylor polynomial.  
In Sections \ref{aestimates} and \ref{bestimates} we deduce estimates for the iterated derivatives of the coefficients before 
and after the splitting. 
Section \ref{specialcover} is dedicated to the proof of Proposition \ref{cover}. 
The proof of Theorem \ref{main} is finally carried out 
in Section \ref{proof}; in Appendix \ref{appendix} we illustrate the proof for polynomials of degree 3 and 4.  
We deduce Theorem \ref{main2} in Section \ref{proofThm2}. 
In Section \ref{applications} we provide three applications of our results: 
local solvability of a system of pseudo-differential equations, a lifting theorem for mappings into orbit spaces of 
finite group representations, and a sufficient condition for multi-valued functions to be of Sobolev class $W^{1,p}$ in the 
sense of Almgren \cite{Almgren00}.

\subsection*{Acknowledgement}

We thank the anonymous referees for the helpful remarks to improve the 
presentation.

\section{Function spaces} \label{functionspaces}

In this section we fix notation for function spaces and recall well-known facts. 

\subsection{H\"older spaces}

Let $\Om \subseteq \R^n$ be open and bounded. We denote by $C^0(\Om)$ the space of continuous complex valued functions on $\Om$.
For $k \in \N \cup \{\infty\}$ we set 
\begin{align*}
  C^k(\Om) &= \{f \in \C^\Om : \p^\al f \in C^0(\Om), 0 \le |\al| \le k\},\\
  C^k(\overline \Om) &= \{f \in C^k(\Om) : \p^\al f \text{ has a continuous extension to } \overline \Om, 
  0 \le |\al| \le k\}.
\end{align*}
For $\al \in (0,1]$ a function $f : \Om \to \C$ belongs to $C^{0,\al}(\overline \Om)$ if it is $\al$-H\"older continuous 
in $\Om$, i.e.,
\[
\Hoeld_{\al,\Om}(f) := \sup_{x,y \in \Om, x \ne y} \frac{|f(x)-f(y)|}{|x-y|^\al} < \infty.
\]
If $f$ is Lipschitz, i.e., $f \in C^{0,1}(\ol \Om)$, we use 
\[
\Lip_\Om (f) =\Hoeld_{1,\Om}(f).
\]
We define 
\[
 C^{k,\al}(\overline \Om) = \{f \in C^k(\overline \Om) : \p^\be f \in C^{0,\al}(\overline \Om), |\be|=k\}.
\]
Note that $C^{k,\al}(\overline \Om)$ is a Banach space when provided with the norm
\[
\|f\|_{C^{k,\al}(\overline \Om)} 
:= \sup_{\substack{|\be| \le k\\ x \in \Om}} |\p^\be f(x)| + \sup_{|\be|=k} \Hoeld_{\al,\Om}(\p^\be f).
\]

\subsection{Lebesgue spaces and weak Lebesgue spaces}  \label{Lebesgue}

Let $\Om \subseteq \R^n$ be open, and let $1 \le p \le \infty$.
We denote by $L^p(\Om)$ the Lebesgue space with respect to the $n$-dimensional Lebesgue measure $\cL^n$. 
For Lebesgue measurable sets $E \subseteq \R^n$ we denote by 
\[
  |E| = \cL^n(E)
\]
its $n$-dimensional Lebesgue measure. Let $p'$ denote the conjugate exponent of $p$ defined by 
\[
  \frac 1p + \frac 1 {p'} =1
\]
with the convention $1' = \infty$ and $\infty' =1$.

Let $1 \le p < \infty$ and let us assume that $\Om$ is bounded. 
A measurable function $f : \Om \to \C$ belongs to the weak $L^p$-space $L_w^p(\Om)$ if 
\[
\|f\|_{p,w,\Om} := \sup_{r\ge 0}  r\, |\{x \in \Om : |f(x)| > r\}|^{1/p} < \infty.
\]  
For $1 \le q < p < \infty$ we have (cf.\ \cite[Ex.\ 1.1.11]{Grafakos08})
\begin{equation} \label{eq:qp}
  \|f\|_{q,w,\Om} \le \|f\|_{L^q(\Om)} \le \Big(\frac{p}{p-q}\Big)^{1/q} 
  |\Om|^{1/q-1/p} \|f\|_{p,w,\Om}
\end{equation}
and hence
$L^p(\Om) \subseteq L_w^p(\Om) \subseteq L^q(\Om) \subseteq L_w^q(\Om)$
with strict inclusions. 
It will be convenient to \emph{normalize} the $L^p$-norm and the $L^p_w$-quasinorm, i.e., we will consider 
\begin{align*}
  \|f\|^*_{L^p(\Om)} &:= |\Om|^{-1/p} \|f\|_{L^p(\Om)},\\
  \|f\|^*_{p,w,\Om} &:= |\Om|^{-1/p} \|f\|_{p,w,\Om}.
\end{align*}
Note that $\|1\|^*_{L^p(\Om)} = \|1\|^*_{p,w,\Om} =1$.
Then, for $1 \le q < p < \infty$,
\begin{gather}
  \|f\|^*_{L^q(\Om)} \le \|f\|^*_{L^p(\Om)}, \label{inclusions0}\\
  \|f\|^*_{q,w,\Om} \le \|f\|^*_{L^q(\Om)} \le \Big(\frac{p}{p-q}\Big)^{1/q} 
   \|f\|^*_{p,w,\Om}. \label{inclusions}
\end{gather}
We remark that $\|\cdot\|_{p,w,\Om}$ is only a quasinorm: the triangle inequality fails, but for  
$f_j \in L_w^p(\Om)$ 
we still have
\[
\Big\|\sum_{j=1}^m f_j \Big\|_{p,w,\Om} \le m \sum_{j=1}^m \|f_j\|_{p,w,\Om}.
\]
There exists a norm equivalent to $\|\cdot\|_{p,w,\Om}$ which makes $L_w^p(\Om)$ into a Banach space if $p>1$. 

The $L^p_w$-quasinorm is $\si$-subadditive: if $\{\Om_j\}$ is a countable family of open sets with 
$\Om = \bigcup \Om_j$ then 
\begin{equation}
  \|f\|^p_{p,w,\Om}  \le \sum_j \|f\|^p_{p,w,\Om_j} \quad \text{ for every } f \in L^p_w(\Om).
\end{equation}
But it is not $\si$-additive:
for instance, for $h : (0,\infty) \to \R$, $h(t):= t^{-1/p}$, we have  
$\|h\|_{p,w,(0,\ep)}^p = 1$ for every $\ep>0$,  
but $\|h\|_{p,w,(1,2)}^p = 1/2$.

\subsection{Sobolev spaces}

For $k \in \N$ and $1 \le p \le \infty$ we consider the Sobolev space
\[
  W^{k,p}(\Om) = \{f \in L^p(\Om) : \p^\al f \in L^p(\Om), 0 \le |\al| \le k\},
\]
where $\p^\al f$ denote distributional derivatives, with the norm
\[
	\|f\|_{W^{k,p}(\Om)} := \sum_{|\al| \le k} \|\p^\al f\|_{L^p(\Om)}.
\] 
On bounded intervals $I \subseteq \R$ the Sobolev space $W^{1,1}(I)$ 
coincides with the space $AC(I)$ of absolutely continuous functions on $I$
if we identify each $W^{1,1}$-functions with its unique continuous representative. 
Recall that a function $f : \Om \to \R$ on an open subset $\Om \subseteq \R$ 
is absolutely continuous if for every $\ep>0$ there exists $\de>0$ so that  
$\sum_{i=1}^n |a_i -b_i| < \de$ implies $\sum_{i=1}^n |f(a_i) -f(b_i)| < \ep$ whenever $[a_i,b_i]$, $i =1,\ldots,n$, 
are non-overlapping intervals contained in $\Om$.

We shall also use $W^{k,p}_{\on{loc}}$, $AC_{\on{loc}}$, etc.\ with the obvious meaning.

\subsection{Extension lemma}

We will use the following extension lemma. The analogue for the $L^p_w$-quasinorm 
may be found in \cite[Lemma 2.1]{ParusinskiRainerAC}
which is a slight generalization of \cite[Lemma 3.2]{GhisiGobbino13}. 
Here we need a version for the $L^p$-norm; 
the proof is the same.  

\begin{lemma} \label{lem:extend}
  Let $\Om \subseteq \R$ be open and bounded, let $f : \Om \to \C$ be continuous, 
  and set $\Om_0 := \{t \in \Om : f(t) \ne 0\}$.  
  Assume that $f|_{\Om_0} \in AC_{\on{loc}}(\Om_0)$ and that $f|_{\Om_0}' \in L^p(\Om_0)$ for some $p\ge 1$ 
  (note that $f$ is differentiable a.e.\ in $\Om_0$).
  Then the distributional derivative of $f$ in $\Om$ is a measurable function $f' \in L^p(\Om)$ and 
  \begin{equation} \label{eq:extend}
  \|f'\|_{L^p(\Om)} = \|f|_{\Om_0}'\|_{L^p(\Om_0)}.
  \end{equation}
\end{lemma}

\begin{proof}
One shows that 
\[
\ps(t) := 
\begin{cases}
  f'(t) & \text{ if } t \in \Om_0,\\
  0     & \text{ if } t \in \Om \setminus \Om_0,  
\end{cases}
\]
represents the distributional derivative of $f$ in $\Om$; for details see \cite[Lemma 2.1]{ParusinskiRainerAC}.
\end{proof}


\section{Radicals of differentiable functions} \label{radicals}

We derive an analogue of Theorem \ref{GhisiGobbino} for complex valued functions.

\begin{proposition} \label{prop:radicals}
  Let $I \subseteq \R$ be a bounded interval, let $k \in \N_{>0}$, and $\al \in (0,1]$.
  For each $g \in C^{k,\al}(\overline I)$ we have 
  \begin{equation} \label{est1}
    |g'(t)| \le \La_{k+\al}(t) |g(t)|^{1-1/(k+\al)}, \quad \text{ a.e.\ in } I,
  \end{equation}  
  for some $\La_{k+\al}=\La_{k+\al,g} \in L_w^p(I,\R_{\ge0})$, 
  where $p =(k+\al)'$, and such that
  \begin{equation} \label{est2}
    \|\La_{k+\al}\|_{p,w,I} \le C(k) \max\Big\{\big(\Hoeld_{\al,I}(g^{(k)})\big)^{1/(k+\al)}|I|^{1/p}, 
    \|g'\|_{L^\infty(I)}^{1/(k+\al)}\Big\}.
  \end{equation}
\end{proposition}

\begin{proof}
  Analogous to the proof of \cite[Proposition~3.1]{ParusinskiRainerAC}.
\end{proof}

\begin{corollary} \label{cor:radicals}
  Let $n$ be a positive integer and 
  let $I \subseteq \R$ be an open bounded interval. Assume that $f : I \to \C$ is a continuous function such that 
  $f^n = g \in C^{n-1,1}(\oI)$.
  Then we have $f' \in L^{n'}_w(I)$ and 
  \begin{equation} \label{est}
    \|f'\|_{n',w,I} \le 
    C(n) \max\Big\{\big(\Lip_{I}(g^{(n-1)})\big)^{1/n}|I|^{1/n'}, 
    \|g'\|_{L^\infty(I)}^{1/n}\Big\}. 
  \end{equation}  
\end{corollary}

\begin{proof}
  On the set $\Om_0 = \{t \in I : f(t) \ne 0\}$, $f$ is differentiable and satisfies
  \[
    |f'(t)|   
    = \frac 1 n \frac{|g'(t)|}{|g(t)|^{1-1/n}}.
  \]  
  So the assertion follows from Proposition \ref{prop:radicals} and the $L^p_w$-analogue of Lemma \ref{lem:extend}; see 
  \cite[Lemma~2.1]{ParusinskiRainerAC}.
\end{proof}

\begin{remark}
  Proposition~\ref{prop:radicals} and hence also Corollary \ref{cor:radicals} are optimal in the following sense:
  \begin{itemize}
    \item $\La_{k+\al}$ can in general not be chosen in $L^p$. Indeed, for $g : (-1,1) \to \R$, $g(t)=t$, we have 
    $|g'| |g|^{1/(k+\al) - 1} = |t|^{-1/p}$ 
    which is not $p$-integrable near $0$; see \cite[Example 4.3]{GhisiGobbino13}.
    \item If $g$ is only in $C^{k,\be}(\overline I)$ for every $\be<\al$, 
    then \eqref{est1} in general fails even for $\La_{k+\al} \in L^1(I)$.
    We refer to \cite[Example 4.4]{GhisiGobbino13} for a non-negative function 
    $g\in \bigcap_{\be<\al} C^{k,\be}(\overline I) \cap C^{\infty}(I)$ and $g \not\in C^{k,\al}(\overline I)$ 
    whose non-negative $(k+\al)$-root has unbounded variation in $I$.    
  \end{itemize}
\end{remark}

\section{Preliminaries on polynomials} \label{polynomials}

\subsection{Tschirnhausen transformation}

A monic polynomial 
\[
  P_a(Z) = Z^n + \sum_{j=1}^n a_j Z^{n-j}, \quad a=(a_1,\ldots,a_n) \in \C^n,
\] 
is said to be in \emph{Tschirnhausen form} if $a_1=0$.
Every polynomial $P_a$ can be transformed to a polynomial $P_{\tilde a}$ in Tschirnhausen form 
by the substitution $Z \mapsto Z-a_1/n$, which we refer to as the 
\emph{Tschirnhausen transformation}, 
\[
  P_{\tilde a}(Z) = P_a(Z-a_1/n) = Z^n + \sum_{j=2}^n \tilde a_j Z^{n-j}, 
  \quad \tilde a=(\tilde a_2,\ldots,\tilde a_n) \in \C^{n-1}.
\]
We have the formulas
\begin{equation}\label{Tschirnhausen}
  \tilde a_j = \sum_{\ell=0}^j C_\ell\, a_\ell\, {a_1}^{j-\ell}, \quad j= 2,\ldots,n, 
\end{equation} 
where $C_\ell$ are universal constants.
The effect of the Tschirnhausen transformation will always be indicated by adding \emph{tilde} to the coefficients, 
$P_a \leadsto P_{\tilde a}$.  

We will identify the set of monic complex polynomials $P_a$ of degree $n$ with the set $\C^n$ (via $P_a \mapsto a$) and 
the set of monic complex polynomials $P_{\tilde a}$ of degree $n$ in Tschirnhausen form 
with the set $\C^{n-1}$ (via $P_{\tilde a} \mapsto \tilde a$).

\subsection{Splitting} \label{ssec:split}

The following well-known lemma (see e.g.\ \cite{AKLM98} or \cite{BM90}) is a consequence of the 
inverse function theorem. 

\begin{lemma} \label{split}
Let $P_a = P_b P_c$, where $P_b$ and $P_c$ are monic complex polynomials without common root.
Then for $P$ near $P_a$ we have $P = P_{b(P)} P_{c(P)}$
for analytic mappings of monic polynomials $P \mapsto b(P)$ and $P \mapsto c(P)$,
defined for $P$ near $P_a$, with the given initial values.
\end{lemma}

\begin{proof}
The splitting $P_a = P_b P_c$ defines on the coefficients a polynomial mapping $\vh$ such that $a = \vh(b,c)$, 
where $a=(a_i)$, $b=(b_i)$, and $c=(c_i)$. The Jacobian determinant  
$\det d\vh(b,c)$ equals the resultant of $P_b$ and $P_c$ which is non-zero by assumption. 
Thus $\vh$ can be inverted locally. 
\end{proof}

If $P_{\tilde a}$ is in Tschirnhausen form and if $\tilde a \ne 0$, then $P_{\tilde a}$ \emph{splits}, i.e., 
$P_{\tilde a} = P_b P_c$ for monic polynomials $P_b$ and $P_c$ with positive degree and without common zero. 
For, if $\la_1,\ldots,\la_n$ denote the roots of $P_{\tilde a}$ and they all coincide, then since 
\[
    \la_1+\cdots+\la_n = \tilde a_1 = 0
\] 
they all must vanish, contradicting $\tilde a \ne 0$.

Let $\tilde a_2,\ldots,\tilde a_n$ denote the coordinates in $\C^{n-1}$ 
($=$ set of polynomials of degree $n$ in Tschirnhausen form).
Fix $k \in \{2,\ldots,n\}$ and let $\tilde p \in \C^{n-1} \cap \{\tilde a_k \ne 0\}$; $\tilde p$ corresponds to the polynomial $P_{\tilde a}$. 
We associate the polynomial 
\begin{gather*}
  Q_{\underline a}(Z) := \tilde a_k^{- n/k} P_{\tilde a} (\tilde a_k^{1/k} Z) 
  = Z^n + \sum_{j=2}^n \tilde a_k^{- j/k} \tilde a_j Z^{n-j},\\
  \underline a_j := \tilde a_k^{- j/k} \tilde a_j, \quad j = 2,\ldots, n, 
\end{gather*}
where some branch of the radical is fixed.  
Then $Q_{\underline a}$ is in Tschirnhausen form and $\underline a_k =1$;
it corresponds to a point $\underline p \in \C^{n-1} \cap \{\underline a_k =1\}$. 
By Lemma~\ref{split} we have a splitting $Q_{\underline a} = Q_{\underline b} Q_{\underline c}$ on some open ball 
$B_\rh(\underline p)$ 
centered at $\underline p$ with radius $\rh>0$. 
In particular, there exist analytic functions $\ps_i$ on $B_\rh(\underline p)$ such that
\begin{equation*} 
  \underline b_i = 
  \ps_i \big(\tilde a_k^{-2/k} \tilde a_2, \tilde a_k^{-3/k} \tilde a_3, \ldots, \tilde a_k^{-n/k} \tilde a_n\big), 
    \quad i = 1,\ldots,\deg Q_{\underline b}.
\end{equation*}
The splitting $Q_{\underline a} = Q_{\underline b} Q_{\underline c}$ induces a splitting 
$P_{\tilde a} = P_b P_c$, where 
\begin{equation} \label{eq:bj}
    b_i =  \tilde a_k^{i/k} 
    \ps_i \big(\tilde a_k^{-2/k} \tilde a_2, \tilde a_k^{-3/k} \tilde a_3, \ldots, \tilde a_k^{-n/k} \tilde a_n\big), 
    \quad i = 1,\ldots,n_b := \deg P_b;
\end{equation}
likewise for $c_j$. 
Shrinking $\rh$ slightly, we may assume that $\ps_i$ and all its partial derivatives are bounded on $B_\rh(\underline p)$. 
Let $\tilde b_j$ denote the coefficients of the polynomial $P_{\tilde b}$ resulting from $P_b$ by the  
Tschirnhausen transformation. Then, by \eqref{Tschirnhausen}, 
\begin{equation} \label{eq:tildebj}
    \tilde b_i = \tilde a_k^{i/k} 
    \tilde \ps_i \big(\tilde a_k^{-2/k} \tilde a_2, \tilde a_k^{-3/k} \tilde a_3, \ldots, \tilde a_k^{-n/k} \tilde a_n\big), 
    \quad i = 2,\ldots,n_b,
\end{equation}
for analytic functions $\tilde \ps_i$ which, together with all their partial derivatives, are bounded on $B_\rh(\underline p)$.

\subsection{Universal splitting of polynomials in Tschirnhausen form} \label{universal}

The set 
\begin{equation} \label{eq:compactK}
  K:= \bigcup_{k=2}^n\{(\ul a_2, \ldots,\ul a_n) \in \C^{n-1} : \ul a_k=1,~ |\ul a_j| \le 1 \text{ for }j \ne k\}
\end{equation}
is compact. For each point $\underline p \in K$
there exists $\rh(\underline p)>0$ such that we have a splitting  
$P_{\tilde a} = P_b P_{c}$ on the open ball $B_{\rh(\underline p)}(\underline p)$, and we fix this splitting; 
cf.\ Section \ref{ssec:split}.
Choose a finite subcover of $K$ by open balls $B_{\rh_\de}(\underline p_\de)$, $\de \in \De$. 
Then there exists $\rh>0$ such that for every $\underline p \in K$ there is a $\de \in \De$ such that 
$B_\rh(\underline p) \subseteq B_{\rh_\de}(\underline p_\de)$.

To summarize, for each integer $n \ge 2$ we have fixed 
\begin{itemize}
   \item a finite cover $\cB$ of $K$ by open balls $B$,
   \item a splitting $P_{\tilde a} = P_b P_{c}$ on each $B \in \cB$ together with analytic functions $\ps_i$ and 
   $\tilde \ps_i$ which are bounded on $B$ along with all their partial derivatives,
   \item a positive number $\rh$ such that for each $\underline p \in K$ there is a $B \in \cB$ such that 
   $B_\rh(\underline p) \subseteq B$ (note that $2\rh$ is a Lebesgue number of the cover $\cB$).  
 \end{itemize} 
We will refer to this data as a \emph{universal splitting of polynomials of degree $n$ in Tschirnhausen form}
and to $\rh$ as the \emph{radius of the splitting}.

\subsection{Coefficient estimates}

The following estimates are crucial. (Here it is convenient to number the coefficients in reversed order.)

\begin{lemma} \label{lem:interpol}
  Let $m\ge 1$ be an integer and $\al \in (0,1]$.
  Let $P(x) = a_1 x + \cdots + a_{m} x^{m} \in \C[x]$ satisfy 
  \begin{equation} \label{interpol1}
    |P(x)| \le A (1+M x^{m+\al}), \quad \text{ for }~ x \in [0,B] \subseteq \R,
  \end{equation}
  and constants $A, M \ge 0$ and $B>0$. Then 
  \begin{equation} \label{interpol2}
    |a_j| \le C A (1+M^{j/(m+\al)} B^j) B^{-j}, \quad j=1,\ldots,m,
  \end{equation}
  for a constant $C$ depending only on $m$ and $\al$.
\end{lemma}

\begin{proof}
The statement is well-known if $M=0$; see \cite[Lemma 3.4]{ParusinskiRainerHyp}. Assume that $M>0$.

It suffices to consider the special case $A=B=1$.  
The general case follows by applying the special case to $Q(x)= A^{-1} P(Bx) =  b_1 x + \cdots + b_{m} x^{m}$, where 
$b_i= A^{-1} B^i  {a_i}$.  

Fix $k \in \{1,\ldots,m\}$ and write the inequality \eqref{interpol1} in the form 
\begin{equation} \label{interpol3}
  |x^{-k} P(x)| \le x^{-k} + M x^{m+\al-k}.
\end{equation}
The function on the right-hand side of \eqref{interpol3} attains is minimum on $\{x>0\}$ at the point 
\begin{equation} \label{interpol5}
  x_k = \big(\tfrac k{m+\al-k}\big)^{1/(m+\al)} M^{-1/(m+\al)},  
\end{equation}
and this minimum is of the form $C_k M^{k/(m+\al)}$ for some $C_k$ depending only on $k$, $m$, and $\al$.
Thus, provided that $x_k \le 1$, we get 
\begin{equation} \label{interpol4}
  |P(x_k)| \le \tilde C_k, 
\end{equation}
for some $\tilde C_k$ depending only on $k$, $m$, and $\al$.

Suppose first that $x_k \le 1$ for all $k= 1,\ldots,m$ and consider 
\begin{equation*}
    a_1 x_k + \cdots + a_{m} x_k^{m} = P(x_k), \quad k= 1,\ldots,m, 
\end{equation*}
as a system of linear equations with the unknowns $a_j M^{-j/(m+\al)}$ and the (Vandermonde-like) matrix 
\begin{equation*}
  L = \big(\big(\tfrac{k}{m+\al-k}\big)^{j/(m+\al)}\big)_{k,j=1}^{m}.   
\end{equation*} 
Then the vector of unknowns is given by 
\[
  (a_1 M^{-1/(m+\al)}, \ldots , a_{m} M^{-m/(m+\al)})^T = L^{-1}  (P(x_1), P(x_2), \ldots , P(x_{m}))^T.
\] 
By \eqref{interpol4}, we may conclude that 
\[
  |a_j| \le C M^{j/(m+\al)}, \quad j=1,\ldots,m,
\]
for a constant $C$ depending only on $m$ and $\al$, that is \eqref{interpol2}.

If $x_k > 1$ then $M < k/(m+\al-k)$, by \eqref{interpol5}. Hence, using \eqref{interpol1}, for $x \in [0,1]$,
\[
  |P(x)| \le 1+ M x^{m+\al} \le 1+ \tfrac k {m+\al-k} \le \tfrac{m+\al}{\al}. 
\]
In this case we may apply the lemma with $M=0$, $A = (m+\al)/\al$, and $B=1$, and obtain 
\[
  |a_j| \le C , \quad j=1,\ldots,m,
\]
for a constant $C$ depending only on $m$ and $\al$, which implies \eqref{interpol2}.
\end{proof}

As a consequence we get estimates for the intermediate derivatives of a finitely differentiable function in terms of the 
function and its highest derivative. 
For an interval $I \subseteq \R$ and a function $f : I \to \C$ we define
\[
  V_I(f) := \sup_{t,s \in I} |f(t)-f(s)|.
\]

\begin{lemma} \label{taylor} 
  Let $I \subseteq \R$ be a bounded open interval, $m \in \N_{>0}$, and $\al \in (0,1]$.
  If $f\in C^{m,\al}(\overline I)$, then for all $t\in I$ 
  and  $s = 1,\ldots,m$,  
    \begin{align}\label{eq:1}  
    |f^{(s)}(t) | \le C |I|^{-s} \bigl(V_I(f) + V_I(f)^{(m+\al-s)/(m+\al)} (\Hoeld_{\al,I}(f^{(m)}))^{s/(m+\al)}  |I|^s
    \bigr),  
  \end{align}
  for a universal constant $C$ depending only on $m$ and $\al$.
\end{lemma}

\begin{proof}
We may suppose that $I=(-\delta, \delta)$. If $t \in I$ then at least one of the two intervals
 $[t,t\pm \delta )$, say $[t,t+ \delta )$, is included in $I$.  
 By Taylor's formula, for $t_1\in [t,t +\delta )$, 
\begin{align*}
      \sum_{s=1}^{m}  \frac{{f}^{(s)}(t)}{s!} (t_1-t)^s 
       & = f(t_1)-f(t) 
       -  \int_0^1 \frac {(1-\ta)^{m-1}}{(m-1)!} \big(f^{(m)}(t+\ta(t_1-t))-f^{(m)}(t)\big)\, d\ta\, (t_1-t)^m      
\end{align*}
and hence
    \begin{align*}
      \Big|\sum_{s=1}^{m}  \frac{{f}^{(s)}(t)}{s!} (t_1-t)^s\Big| 
      & \le V_I(f)+  \Hoeld_{\al,I}(f^{(m)})  (t_1-t)^{m+\al}
      \\
      & = V_I(f)\big(1 + V_I(f)^{-1} \Hoeld_{\al,I}(f^{(m)})  (t_1-t)^{m+\al} \big).        
    \end{align*}
The assertion follows from Lemma~\ref{lem:interpol}.   
\end{proof}

\subsection{Higher order Glaeser inequalities} \label{Glaeser}

As a corollary of Lemma \ref{taylor} we obtain a generalization of 
Ghisi and Gobbino's higher order Glaeser inequalities \cite[Proposition~3.4]{GhisiGobbino13}.

\begin{corollary}
  Let $m \in \N_{>0}$ and $\al \in (0,1]$.
  Let $I = (t_0- \de,t_0+\de)$ with $t_0 \in \R$ and $\de>0$.
  If $f\in C^{m,\al}(\overline I)$ is such that $f$ and $f'$ do not change their sign on $I$,
  then for all $s = 1,\ldots,m$,
  \begin{align}\label{eq:2}  
    |f^{(s)}(t_0) | \le C |I|^{-s} \bigl(|f(t_0)| + |f(t_0)|^{(m+\al-s)/(m+\al)} (\Hoeld_{\al,I}(f^{(m)}))^{s/(m+\al)}  |I|^s
    \bigr),  
  \end{align}
  for a universal constant $C$ depending only on $m$ and $\al$.
\end{corollary}

\begin{proof}
  For simplicity assume $t_0=0$.
  Changing $f$ to $-f$ and $t$ to $-t$ if necessary, we may assume that $f(t) \ge 0$ and $f'(t)\le 0$ for all $t \ge 0$.
  Then $V_{[0,\de)}(f) \le f(0)$ and so \eqref{eq:2} follows from \eqref{eq:1}.
\end{proof}

For $s=1$ we recover \cite[Proposition~3.4]{GhisiGobbino13}. Indeed, 
for $s =1$ we may write \eqref{eq:2} as 
\begin{align} \label{eq:3}   
    |f'(t_0) | \le 
    C  |f(t_0)|^{(m+\al-1)/(m+\al)} \max\bigl\{|f(t_0)|^{1/(m+\al)} |I|^{-1}, (\Hoeld_{\al,I}(f^{(m)}))^{1/(m+\al)}\bigr\},  
\end{align}
and the inequality in \cite[Proposition~3.4]{GhisiGobbino13} can be written as 
\begin{align} \label{eq:4}  
    |f'(t_0) | \le C  |f(t_0)|^{(m+\al-1)/(m+\al)} 
    \max\bigl\{|f'(t_0)|^{1/(m+\al)} |I|^{-1+1/(m+\al)}, (\Hoeld_{\al,I}(f^{(m)}))^{1/(m+\al)}  \bigr\}.  
\end{align}
These two inequalities are equivalent in the following sense:  
 if \eqref{eq:3} holds with the constant $C>0$ then \eqref{eq:4} holds with the constant 
 $\max \{C,C^{(m+\al-1)/(m+\al)}\}$, and, 
 symmetrically,  if \eqref{eq:4} holds with the constant $C>0$ then \eqref{eq:3} holds with the constant 
 $\max \{C,C^{(m+\al)/(m+\al-1)}\}$.  For instance, suppose that \eqref{eq:3} holds.  
 If the second term in the maximum (in \eqref{eq:3}) is dominant, 
 then \eqref{eq:4} holds with the same constant.  If the first term is dominant in the 
 maximum, that is $|f'(t_0)| \le C |f(t_0)| |I|^{-1}$, then $|f'(t_0)| ^{(m+\al-1)/(m+\al)} \le (C
  |f(t_0)| |I|^{-1})^{(m+\al-1)/(m+\al)}$ and \eqref{eq:4} holds 
  with the constant  $C^{(m+\al-1)/(m+\al)}$.

\section{Estimates for the iterated derivatives of the coefficients} \label{aestimates}

In the next three sections we collect the necessary tools for the proof of Theorem \ref{main}. 
In the current section we derive estimates for the derivatives of the coefficients of a 
$C^{n-1,1}$-curve of polynomials of degree $n$ in Tschirnhausen form.    

\subsection{Preparations for the splitting}

Let $I\subseteq \R$ be a bounded open interval and let 
\begin{equation} \label{eq:polynomial}
  P_{\tilde a(t)}(Z) = Z^n + \sum_{j=2}^n \tilde a_j(t) Z^{n-j},\quad t \in I,    
\end{equation} 
be a monic complex polynomial in Tschirnhausen form with coefficients $\tilde a_j \in C^{n-1,1}( \oI )$, $j = 2,\ldots,n$.
We make the following assumptions.
Suppose that $t_0\in I$ and $k \in \{2,\dots,n\}$ are such that 
\begin{equation} \label{eq:k}
  |\tilde a_k(t_0)|^{1/k} =  \max_{2 \le j\le n} |\tilde a_j(t_0)|^{1/j}\ne 0  
 \end{equation} 
and that, for some positive constant $B<  1/3$,  
\begin{align}\label{assumpt}
\sum_{j=2}^n \|(\tilde a_j^{1/j})'\|_{L^1 (I)} \le  B |\tilde a_k(t_0)|^{1/k} .
\end{align}

By Corollary~\ref{cor:radicals}, 
every continuous selection $f$ of 
the multi-valued function $\tilde a_j^{1/j}$ is absolutely continuous on $I$, and $\|f'\|_{L^1(I)}$ is independent of the 
choice of the selection (by \eqref{eq:extend}).  
(By a selection of a set-valued function 
$F : X \leadsto Y$ we mean a single-valued function $f : X \to Y$ such that $f(x) \in F(x)$ 
for all $x \in X$.)
So henceforth we shall fix one continuous selection of $\tilde a_j^{1/j}$ and, abusing notation, 
denote it by $\tilde a_j^{1/j}$ as well.

\begin{lemma} \label{lem1}
  Assume that the polynomial \eqref{eq:polynomial} satisfies \eqref{eq:k}--\eqref{assumpt}. 
  Then for all $t \in I$ and $j=2,\ldots,n$,
  \begin{align} \label{eq:ass10} 
    |\tilde a_j^{1/j} (t) - \tilde a_j^{1/j} (t_0)| \le  B |\tilde a_k (t_0)|^{1/k},  
  \end{align}
  \begin{align} \label{eq:ass11}  
    \frac 2 3 <    1-B \le \Big | \frac{\tilde a_k(t)}{\tilde a_k(t_0)} \Big |^{1/k}   \le 1+B < \frac4 3,  
  \end{align}
  \begin{equation} \label{eq:ass12}
    |\tilde a_j(t)|^{1/j}\le \frac 4 3 |\tilde a_k(t_0)|^{1/k} \le  2  |\tilde a_k(t)|^{1/k}.
  \end{equation}
\end{lemma}

\begin{proof}
  First, \eqref{eq:ass10} is a consequence of \eqref{assumpt},
  \begin{align*}
    |\tilde a_j^{1/j} (t) - \tilde a_j^{1/j} (t_0)| = |\int_{t_0}^t (\tilde a_j^{1/j})' \,ds| 
    \le \| (\tilde a_j^{1/j})'\|_{L^1(I)} \le B |\tilde a_k(t_0)|^{1/k}. 
  \end{align*}
  For $j=k$ it implies
  \begin{align*}
      \Big|\Big | \frac{\tilde a_k(t)}{\tilde a_k(t_0)}\Big |^{1/k} -1\Big| \le  B, 
  \end{align*}
  and thus \eqref{eq:ass11}.
  By \eqref{eq:k}, \eqref{eq:ass10}, and \eqref{eq:ass11}, 
  \[
    |\tilde a_j(t)|^{1/j} \le (1+B) |\tilde a_k(t_0)|^{1/k} \le 2 |\tilde a_k(t)|^{1/k},
  \]
  that is \eqref{eq:ass12}.
\end{proof}

By \eqref{eq:ass11}, $\tilde a_k$ does not vanish on the interval $I$ and so the curve
\begin{align} \label{curve}
    \underline a : I &\to \{(\underline a_2,\dots,\underline a_n) \in \C^{n-1} : \underline a_k=1\} \\
    t &\mapsto \ul a(t) := (\tilde a_k^{-2/k} \tilde a_2, \ldots, \tilde a_k^{-n/k} \tilde a_n)(t) \notag
\end{align}
is well-defined.

\begin{lemma} \label{lem2}
  Assume that the polynomial \eqref{eq:polynomial} satisfies \eqref{eq:k}--\eqref{assumpt}. 
  Then the length of the curve \eqref{curve} 
  is bounded by $3 n^2\, 2^{n} B$.
\end{lemma}

\begin{proof} 
  The estimates \eqref{eq:ass10}, \eqref{eq:ass11}, and \eqref{eq:ass12} 
  imply 
  \begin{align*}
    |\tilde a_k^{- j/k} \tilde a_j'| &\le  2^n |\tilde a_j^{- 1+1/j } \tilde a_j' \tilde a_k^{-1/k}|  
    \le  3n\, 2^{n-1} |(\tilde a_j^{1/j})'|   |\tilde a_k (t_0)|^{-1/k}    
    \\
   | (\tilde a_k^{-j/k})' \tilde a_j| &\le n 2^n |\tilde a_k^{-1/k} (\tilde a_k^{1/k})'| 
   \le  3n\, 2^{n-1} |(\tilde a_k^{1/k})'|  
   |\tilde a_k (t_0)|^{-1/k},
  \end{align*}
  and thus 
  \[
    |(\tilde a_k^{-j/k} \tilde a_j)'| 
    \le 3n\, 2^{n-1} |\tilde a_k (t_0)|^{-1/k} \Big(|(\tilde a_j^{1/j})'| + |(\tilde a_k^{1/k})'|\Big).
  \]
  Consequently, using \eqref{assumpt},
  \begin{align*}
    \int_I |\ul a'| \,ds \le 3 n^2\, 2^{n} B, 
  \end{align*}
  as required.
\end{proof}

\subsection{Estimates for the derivatives of the coefficients}
Let us replace \eqref{assumpt} by the stronger assumption
\begin{align}\label{assumption}
  \const |I| + \sum_{j=2}^n \|(\tilde a_j^{1/j})'\|_{L^1 (I)} \le  B |\tilde a_k(t_0)|^{1/k},
\end{align} 
where 
\begin{equation} \label{M}
  \const = \max_{2 \le j\le n}  (\Lip_I(\tilde a_j^{(n-1)}))^{1/n} |\tilde a_k (t_0)|^{(n-j)/(kn)}.
\end{equation}

\begin{lemma} \label{bounds2}
Assume that the polynomial \eqref{eq:polynomial} satisfies \eqref{eq:k} and \eqref{assumption}. 
Then 
for all  
$ j = 2,\ldots,n$ and  $ s = 1,\ldots,n-1$,  
    \begin{align}\label{est:a} 
    \begin{split}
       \|\tilde a_j^{(s)} \|_{L^\infty(I)} &\le C(n)  |I|^{-s}  |\tilde a_k (t_0)|^{j/k},
    \\
    \Lip_I (\tilde a_j^{(n-1)})  &\le C(n)  |I|^{-n}  |\tilde a_k (t_0)|^{j/k}.
    \end{split}   
  \end{align}
\end{lemma}

\begin{proof}
The second estimate in \eqref{est:a} is immediate from \eqref{assumption}.
Let $t \in I$.
By Lemma \ref{taylor},    
  \begin{align*}  
    |\tilde a_j^{(s)}(t) | \le C |I|^{-s} \bigl( V_I(\tilde a_j) 
    +  V_I(\tilde a_j)^{(n-s)/n} \Lip_I(\tilde a_j^{(n-1)})^{s/n}  |I|^s\bigr).   
  \end{align*}
By \eqref{eq:ass12}, 
\[
  V_I(\tilde a_j) \le 2 \|\tilde a_j\|_{L^{\infty}(I)} \le 2\, (4/3)^n |\tilde a_k (t_0)|^{j/k} 
\]
and, by \eqref{assumption},
\[
   \max_{2 \le j\le n}  (\Lip_I(\tilde a_j^{(n-1)}))^{s/n} |\tilde a_k (t_0)|^{-js/(kn)} |I|^s 
   = |\tilde a_k (t_0)|^{-s/k} \const^s |I|^s \le  1.
\]
Thus 
\begin{align*}
  \MoveEqLeft
  V_I(\tilde a_j) +  V_I(\tilde a_j)^{(n-s)/n} \Lip_I(\tilde a_j^{(n-1)})^{s/n}  |I|^s
  \\
  &\le
   |\tilde a_k (t_0)|^{j/k} \big(C_1 + C_2  \Lip_I(\tilde a_j^{(n-1)})^{s/n} |\tilde a_k (t_0)|^{-js/(kn)} |I|^s \big) 
  \\
  &\le C_3 |\tilde a_k (t_0)|^{j/k}, 
\end{align*}
for constants $C_i$ that depend only on $n$.
So also the first estimate in \eqref{est:a} is proved.  
\end{proof}

\section{The estimates after splitting} \label{bestimates}

In this section we assume that our polynomial splits. 
We prove that the coefficients of each factor of the splitting satisfy estimates analogous to those in \eqref{est:a} 
on suitable subintervals.

\subsection{Estimates after splitting on \texorpdfstring{$I$}{I}}

Assume that the polynomial \eqref{eq:polynomial} satisfies \eqref{eq:k}--\eqref{assumpt} and the estimates \eqref{est:a}.

Additionally,
we suppose   
that the curve $\underline a$ defined in \eqref{curve} lies entirely in one of the balls $B_\rh(\underline p)$ 
from Section \ref{ssec:split} on which we have a splitting. 
Then 
$P_{\tilde a}$ splits on $I$, 
\begin{align}\label{splitting}
  P_{\tilde a}(t) = P_b(t) P_{b^*}(t), \quad t \in I. 
\end{align}
By \eqref{eq:bj} and \eqref{eq:tildebj}, the coefficients $b_i$ are of the form 
\begin{equation} \label{eq:b_i0}
  b_i = \tilde a_k^{i/k} \ps_i \big(\tilde a_k^{-2/k} \tilde a_2, \ldots, \tilde a_k^{-n/k} \tilde a_n\big), 
  \quad i = 1,\ldots, n_b,
\end{equation}
and after the Tschirnhausen transformation $P_b \leadsto P_{\tilde b}$, we get 
  \begin{equation} \label{eq:b_i}
  \tilde b_i = \tilde a_k^{i/k} \tilde \ps_i \big(\tilde a_k^{-2/k} \tilde a_2, \ldots, \tilde a_k^{-n/k} \tilde a_n\big),
  \quad i = 2,\ldots, n_b,
  \end{equation}  
where $\ps_i$ and $\tilde \ps_i$ are the analytic functions specified in Section \ref{ssec:split} and 
$n_b = \deg P_b$.

\begin{lemma} \label{lem:B}
  Assume that the polynomial \eqref{eq:polynomial} satisfies \eqref{eq:k}--\eqref{assumpt}, \eqref{est:a}, 
  and \eqref{splitting}--\eqref{eq:b_i}. 
  Then for all 
  $i=2,\ldots,n_b$ and $s = 1,\ldots,n-1$,
  \begin{align} \label{eq:b_ider}
  \begin{split}
    \|\tilde b_i^{(s)} \|_{L^\infty(I)} &\le C  |I|^{-s}  |\tilde a_k (t_0)|^{i/k}, 
      \\   
      \Lip_I(\tilde b_i^{(n-1)})  &\le C  |I|^{-n}  |\tilde a_k (t_0)|^{i/k},
  \end{split}
  \end{align}  
  where $C$ is a constant depending only on $n$ and on the functions $\tilde \ps_i$.
\end{lemma}

\begin{proof}
  Let us prove the first estimate in \eqref{eq:b_ider}.
  Let $F$ be any $C^n$-function defined on an open set $U$ in $\C^{n-1}$ containing 
  $\underline a(I)$ and satisfying $\|F\|_{C^n(\ol U)} < \infty$. 
  We claim that, for $s = 1,\ldots,n-1$, 
  \begin{align} \label{eq:B2}
    \|\p_t^s (F \o \ul a)\|_{L^\infty(I)} &\le C  |I|^{-s}, 
  \end{align}
  where 
  $C$ is a constant depending only on $n$ and $\|F\|_{C^n(\ol U)}$. 
  For any real exponent $r$, Fa\`a di Bruno's formula implies 
  \begin{equation} \label{FaadiBruno}
    \p_t^{s} \big(\tilde a_j^{r} \big) = 
    \sum_{\ell \ge 1}^{s} \sum_{\ga \in \Ga(\ell,s)} c_{\ga,\ell,r}  
    \, \tilde a_j^{r-\ell} \tilde a_j^{(\ga_1)} \cdots \tilde a_j^{(\ga_\ell)}
  \end{equation}
  where $\Ga(\ell,s) = \{\ga \in \N_{>0}^\ell : |\ga| = s\}$ and
  \[
    c_{\ga,\ell,r}=  \frac{s!}{\ell! \ga!} r(r-1) \cdots (r-\ell+1).
  \]
  By \eqref{est:a} and \eqref{eq:ass11}, this implies
  \begin{align} \label{eq:derpower}
    \|\p_t^{s} \big(\tilde a_j^{r} \big)\|_{L^\infty(I)} 
    &\le  
    \sum_{\ell \ge 1}^{s} \sum_{\ga \in \Ga(\ell,s)} c_{\ga,\ell,r}  
    \, \|\tilde a_j^{r-\ell}\|_{L^\infty(I)} \|\tilde a_j^{(\ga_1)}\|_{L^\infty(I)} 
    \cdots \|\tilde a_j^{(\ga_\ell)}\|_{L^\infty(I)} 
    \notag \\
    &\le  
    C(n) \sum_{\ell \ge 1}^{s} \sum_{\ga \in \Ga(\ell,s)} c_{\ga,\ell,r}  
    \, |\tilde a_k(t_0)|^{(r-\ell)j/k} 
    |I|^{-s} |\tilde a_k(t_0)|^{\ell j/k}
    \notag \\
    &\le  
    C(n)  
    |I|^{-s} |\tilde a_k(t_0)|^{r j/k}.
  \end{align}
  Together with the Leibniz formula,
  \[
    \p_t^s \big(\tilde a_k^{- j/k} \tilde a_j\big) 
    = \sum_{q=0}^{s} \binom{s}{q} \tilde a_j^{(q)} \p_t^{s-q} \big(\tilde a_k^{- j/k} \big),  
  \]
  \eqref{eq:derpower} and \eqref{est:a} lead to 
  \begin{align} \label{eq:underlinea}
     \|\p_t^s \big(\tilde a_k^{- j/k} \tilde a_j\big) \|_{L^\infty(I)}  
    &\le C(n) |I|^{-s}.   
  \end{align}
  Again by the Leibniz formula,
  \begin{align*}
    \p_t (F \o \ul a) 
    &= \sum_{j=2}^n ((\p_{j-1} F)\o \ul a)\, \p_t \big(\tilde a_k^{- j/k} \tilde a_j\big),
    \\
    \p_t^s (F \o \ul a) 
    &= \sum_{j=2}^n \p_t^{s-1}\Big(((\p_{j-1} F)\o \ul a)\, \p_t \big(\tilde a_k^{- j/k} \tilde a_j\big)\Big) 
    \\
    &= \sum_{j=2}^n \sum_{p=0}^{s-1} 
    \binom{s-1}{p} \p_t^{p}((\p_{j-1} F)\o \ul a)\, \p_t^{s- p} \big(\tilde a_k^{- j/k} \tilde a_j\big). 
  \end{align*}  
  For $s=1$ we immediately get \eqref{eq:B2}. For $1< s \le n-1$, we may argue by induction on $s$. By induction hypothesis,
  \[
    \|\p_t^{p}((\p_{j-1} F)\o \ul a)\|_{L^\infty(I)} \le C(n,\|\p_{j-1} F\|_{C^{s}(\overline U)}) |I|^{-p},
  \]
  for $p = 1,\ldots,s-1$. Together with \eqref{eq:underlinea} this entails \eqref{eq:B2}.

  Now the first part of \eqref{eq:b_ider} is a consequence of \eqref{eq:b_i}, \eqref{eq:derpower} (for $j=k$ and $r = i/k$) 
  and \eqref{eq:B2} (applied to $F= \tilde \ps_i$).

For the second part of \eqref{eq:b_ider}
observe that for functions $f_1,\ldots,f_m$ on $I$ we have 
\[
  \Lip_I(f_1 f_2 \cdots f_m) \le \sum_{i=1}^m \Lip_I(f_i) \|f_1\|_{L^{\infty}(I)} \cdots \widehat{\|f_i\|_{L^{\infty}(I)}} 
  \cdots \|f_m\|_{L^{\infty}(I)}.
\]
Applying it to \eqref{FaadiBruno} and using
\begin{equation*}
  \Lip_I(\tilde a_j^{r-\ell}) \le |r-\ell| \| \tilde a_j^{r-\ell-1} \|_{L^\infty(I)} \|\tilde a_j'\|_{L^\infty} 
\end{equation*}
we find, as in the derivation of \eqref{eq:derpower},
\[
  \Lip_I(\p_t^{n-1}(\tilde a_j^r)) \le C(n,r) |I|^{-n} |\tilde a_k(t_0)|^{rj/k}. 
\]
As above this leads to 
\begin{align*}
     \Lip_I(\p_t^{n-1} \big(\tilde a_k^{- j/k} \tilde a_j\big) )  
    &\le C(n) |I|^{-n}.   
  \end{align*}
and
\[
  \Lip_I (\p_t^{n-1}(F \o \ul a)) \le C(n, \|F\|_{C^n(\overline U)})  |I|^{-n},
\]
and finally to the second part of \eqref{eq:b_ider}.
\end{proof}

\begin{remark} \label{rem:b1prime}
  In the setup of Lemma \ref{lem:B} the same estimates hold for $\tilde b_i$ replaced by $b_i$. This follows by the 
  same proof where ones uses \eqref{eq:b_i0} instead of \eqref{eq:b_i}. We shall only need the special case $i=s=1$ which we 
  state explicitly for later reference:
  \begin{equation} \label{b1prime}
      \|b_1' \|_{L^\infty(I)} \le C  |I|^{-1}  |\tilde a_k (t_0)|^{1/k}.   
  \end{equation}    
\end{remark}

\begin{lemma} \label{lem:Lpbak}
  Assume that $\tilde b_i$, $i = 2,\ldots, m$, are $C^{n-1,1}$-functions, 
  where $m\le n$, on an open bounded interval $I$ which satisfy
  \eqref{eq:b_ider} for all $s=1,\ldots,n-1$. Then, for all $1 \le p < m'$,
  \begin{equation} \label{eq:Lpbak}
     \|(\tilde b_i^{1/i})'\|^*_{L^p(I)} \le C |I|^{-1} |\tilde a_k(t_0)|^{1/k},
   \end{equation} 
   for a constant $C$ which depends only on $n$, $p$, and the constant in \eqref{eq:b_ider}.
\end{lemma}

\begin{proof}
  By \eqref{est} and \eqref{eq:b_ider}, 
  \begin{align*}
    \| (\tilde b_i^{1/i})'\|_{i',w,I} &\le C(i) \max   \Big\{\big(\on{Lip}_{I}(\tilde b_i^{(i-1)})\big)^{1/i}|I|^{1/i'}, 
    \|\tilde b_i'\|_{L^\infty(I)}^{1/i}\Big\} \\
    &\le C  |I|^{-1 +1/i'} |\tilde a_k (t_0)|^{1/k},
\end{align*}
or equivalently,
\begin{equation*}
  \|(\tilde b_{i}^{1/i})'\|^*_{i',w,I} \le C  |I|^{-1} |\tilde a_k (t_0)|^{1/k}.
\end{equation*} 
In view of \eqref{inclusions}, this entails \eqref{eq:Lpbak}.
\end{proof}

\subsection{Special subintervals of \texorpdfstring{$I$}{I} and estimates on them} \label{subintervals}

Assume that the polynomial \eqref{eq:polynomial} satisfies \eqref{eq:k}--\eqref{assumpt}, \eqref{est:a}, 
and \eqref{splitting}--\eqref{eq:b_i}.
Suppose that $t_1 \in I$ and $\ell \in \{2,\ldots,n_b\}$ are such that  
\begin{equation} \label{eq:ell}
  |\tilde b_\ell(t_1)|^{1/\ell} = \max_{2 \le i \le n_b} |\tilde b_i(t_1)|^{1 /i} \ne 0.
\end{equation}
By \eqref{eq:ass12} and \eqref{eq:b_i}, for all $t \in I$ and $i = 2,\ldots,n_b$,
      \begin{equation} \label{eq:b_ibound}
|\tilde b_i(t) | \le C_1  |\tilde a_k (t_0)|^{i/k},   
  \end{equation}
where the constant $C_1$ depends only on the functions $\tilde \ps_i$.
Thanks to \eqref{eq:b_ibound} we can choose  
a constant $D< 1/3$ and an open interval $J$ with $t_1 \in J \subseteq I$ such that 
\begin{align} \label{assumption2a}
  | J|  |I|^{-1}  {|\tilde a_k(t_0)|^{1/k}}  
  + \sum_{i=2}^{n_b} \|(\tilde b _i^{1/i})'\|_{L^1 (J)} =  D |\tilde b_\ell(t_1)|^{1/\ell} .
\end{align}
It suffices to take $D < C_1^{-1}$ where $C_1$ is the constant in \eqref{eq:b_ibound}; 
note that $\tilde b _i^{1/i}$ is absolutely continuous 
by Corollary \ref{cor:radicals}. 

\begin{remark}
  The identity \eqref{assumption2a} will be crucial for the proof of Theorem \ref{main}. 
\end{remark}

We will now see that on the interval $J$ the estimates of Section \ref{aestimates} hold for $\tilde b_i$ 
instead of $\tilde a_j$.

\begin{lemma} \label{lemB}
  Assume that the polynomial \eqref{eq:polynomial} satisfies \eqref{eq:k}--\eqref{assumpt}, \eqref{est:a}, 
  \eqref{splitting}--\eqref{eq:b_i}, and \eqref{eq:ell}. 
  Let $D$ and $J$ be as in \eqref{assumption2a}. 
  Then the functions $\tilde b_i$ on $J$ satisfy the conclusions of Lemmas \ref{lem1}, \ref{lem2}, and \ref{bounds2}. 
  More precisely, 
  for all $t \in J$ and $i = 2,\ldots,n_b$,
  \begin{gather}  
    |\tilde b_i^{1/i} (t) - \tilde b_i^{1/i} (t_1)| \le  D |\tilde b_\ell (t_1)|^{1/\ell}, \label{b1} \\ 
    \frac 2 3 <     1-D \le \Big | \frac{\tilde b_\ell(t)}{\tilde b_\ell(t_1)} \Big |^{1/\ell}   \le 1+D < \frac 4 3,  
    \label{b2}\\
    |\tilde b_i(t)|^{1/i}\le \frac 4 3 |\tilde b_\ell(t_1)|^{1/\ell} \le 2 |\tilde b_\ell(t)|^{1/\ell}. \label{b3}
  \end{gather}  
  The length of the curve 
  \begin{equation}
    J \ni t \mapsto \ul b(t) 
    := (\tilde b_\ell^{-2/\ell} \tilde b_2, \ldots, \tilde b_\ell^{-n_b/\ell} \tilde b_{n_b})(t)
  \end{equation}
  is bounded by $3 n_b^2\, 2^{n_b} D$.  
  For all $ i = 2,\ldots,n_b$ and $ s = 1,\ldots,n-1$,  
    \begin{align}\label{b4}
    \begin{split}
      \|\tilde b_i^{(s)} \|_{L^\infty(J)} &\le C |J|^{-s}  |\tilde b_\ell (t_1)|^{i/\ell},
      \\ 
      \Lip_J(\tilde b_i^{(n-1)})  &\le C |J|^{-n}  |\tilde b_\ell (t_1)|^{i/\ell},    
    \end{split}    
  \end{align}
  for a universal constant $C$ depending only on $n$ and $\tilde \ps_i$.
\end{lemma}

\begin{proof}
  The proof of \eqref{b1}--\eqref{b3} is analogous to the proof of Lemma \ref{lem1}; use \eqref{eq:ell} and \eqref{assumption2a} 
  instead of \eqref{eq:k} and \eqref{assumpt}. The bound for the length of the curve $J \ni t \mapsto \ul b(t)$ (which is 
  well-defined by \eqref{b2}) 
  follows from \eqref{assumption2a} and \eqref{b1}--\eqref{b3}; see the proof of Lemma \ref{lem2}. 

  Let us prove \eqref{b4}.
  By \eqref{eq:b_ider}, for $t\in I$ and $i = 2,\ldots,n_b$ (note that $n_b < n$),
  \begin{equation} \label{eq:b_iderj} 
    |\tilde b_i^{(i)} (t) | \le C  |I|^{-i} |\tilde a_k(t_0)|^{i/k}, 
  \end{equation}
  where $C = C(n,\tilde \ps_i)$.
  Thus, for $t\in J$ and $s=1,\ldots, i$,
  \begin{align*}
  |\tilde b_i^{(s)}(t) | 
    &\le C | J |^{-s} \bigl (  V_J(\tilde b_i)  
    + V_J(\tilde b_i)^{(i-s)/i} \|\tilde b_i^{(i)}\|^{s/i}_{L^\infty(J)} |J|^s \bigr) 
    \hspace{21mm} \text{by Lemma \ref{taylor}}\\
    &\le C_1 | J |^{-s} \Bigl (  |\tilde b_\ell(t_1)|^{i/\ell}  
    + |\tilde b_\ell(t_1)|^{(i-s)/\ell} |J|^s |I|^{-s} |\tilde a_k(t_0)|^{s/k} \Bigr) 
    \hspace{5mm} \text{by \eqref{b3} and \eqref{eq:b_iderj}} \\  
  &\le C_2  | J |^{-s}  |\tilde b_\ell(t_1)|^{i/\ell}  \hspace{69mm} \text{by \eqref{assumption2a}}, 
  \end{align*}
  for constants $C = C(i)$ and $C_h = C_h(n,\tilde \ps_i)$.
  For $s>i$ (including $s = n$), we have $(|J||I|^{-1})^s \le (|J||I|^{-1})^i$ and thus
  \begin{align*}
      | I |^{-s}  |\tilde a_k(t_0)|^{i/k}          
    \le    | J |^{-s} \big(|J| |I|^{-1} |\tilde a_k(t_0)|^{1/k} \big)^i  
    \le     | J |^{-s}  |\tilde b_\ell(t_1)|^{i/ \ell},  
  \end{align*}
  where the second inequality follows from \eqref{assumption2a}. Hence \eqref{eq:b_ider} implies \eqref{b4}.   
\end{proof}

\section{A special cover by intervals} \label{specialcover}

In this section we prove a technical result which will allow us to 
glue local $L^p$-estimates to global ones in the proof of Theorem \ref{main}.

\subsection{Intervals of first and second kind} \label{intervals}

Let $I \subseteq \R$ be a bounded open interval.
Let $\tilde b_i \in C^{n_b-1,1}(\overline I)$, $i =2,\ldots,n_b$.
For each point $t_1$ in  
\[
  I' := I \setminus \{t \in I : \tilde b_2 (t) = \cdots = \tilde b_{n_b}(t) = 0\}
\]
there exists $\ell \in \{2, \ldots, n_b\}$ such that \eqref{eq:ell}.
Assume that there are positive constants $D < 1/3$ and $L$ such that 
for all $t_1 \in I'$
there is 
an open interval $J=J(t_1)$ with $t_1 \in J \subseteq I$ such that 
\begin{align} \label{assumption2aG}
  L | J|    
  + \sum_{i=2}^{n_b} \|(\tilde b _i^{1/i})'\|_{L^1 (J)} =  D |\tilde b_\ell(t_1)|^{1/\ell} .
\end{align}
Note that \eqref{eq:ell} and \eqref{assumption2aG} imply \eqref{b2} (cf.\ the proof of Lemma \ref{lemB}); 
in particular, we have $J \subseteq I'$.

Let us consider the functions 
\begin{align*}
  \vh_{t_1,+}(s) &:= L (s-t_1)   
  + \sum_{i=2}^{n_b} \|(\tilde b _i^{1/i})'\|_{L^1 ([t_1,s))}, \quad s\ge t_1, \\
  \vh_{t_1,-}(s) &:= L (t_1-s)   
  + \sum_{i=2}^{n_b} \|(\tilde b _i^{1/i})'\|_{L^1 ((s,t_1])}, \quad s\le t_1. 
\end{align*}
Then $\vh_{t_1,\pm} \ge 0$ are monotonic continuous functions defined for small $\pm (s - t_1)\ge 0$ 
and satisfying $\vh_{t_1,\pm}(t_1) = 0$. 
We let $\vh_{t_1,\pm}$ grow until 
$\vh_{t_1,-}(s_-) + \vh_{t_1,+}(s_+) = D |\tilde b_\ell(t_1)|^{1/\ell}$, that is \eqref{assumption2aG} with $J=(s_-,s_+)$. 
And we do this 
\emph{symmetrically} whenever possible:  
\begin{enumerate}
   \item[(i)] We say that the interval $J=(s_-,s_+)$ is of \emph{first kind} if 
   \begin{equation} \label{1kind}
     \vh_{t_1,-}(s_-) = \vh_{t_1,+}(s_+) = \frac D 2|\tilde b_\ell(t_1)|^{1/\ell}.
   \end{equation}
   \item[(ii)] If \eqref{1kind} is not possible, i.e., 
   we reach the boundary of the interval $I$ before either $\vh_{t_1,-}$ or $\vh_{t_1,+}$ has 
   grown to the value $(D/2)|\tilde b_\ell(t_1)|^{1/\ell}$, then we say that $J=(s_-,s_+)$ is of \emph{second kind}.  
\end{enumerate} 

\begin{remark}
	We may always assume that the interval $J(t_1)$ if of first kind, if such a choice for $t_1$ exists.
\end{remark}

\subsection{A special subcover} 

The goal of this section is to prove the following proposition.

\begin{proposition} \label{cover}
  Let $I \subseteq \R$ be a bounded open interval.
  Let $\tilde b_i \in C^{n_b-1,1}(\overline I)$, $i =2,\ldots,n_b$.
  For each point $t_1$ in $I'$ fix $\ell \in \{2, \ldots, n_b\}$ such that \eqref{eq:ell}.
  Let $\{J(t_1)\}_{t_1\in I'}$ be a collection of open intervals $J=J(t_1)$ with $t_1 \in J \subseteq I$ such that:
  \begin{enumerate}
    \item There are positive constants $D < 1/3$ and $L$ such that 
      for all $t_1 \in I'$ we have \eqref{assumption2aG} for $J=J(t_1)$.
    \item The interval $J(t_1)$ is of first kind, i.e., \eqref{1kind} holds, if such a choice for $t_1$ exists.
  \end{enumerate}
Then the collection $\{J(t_1)\}_{t_1\in I'}$ has a countable subcollection $\cJ$ that still 
  covers $I'$ and such that every point in $I'$ belongs to at most two intervals in $\cJ$. 
  In particular, 
  \[
    \sum_{J \in \cJ} |J| \le 2 |I'|.
  \]
\end{proposition}

\begin{remark}
   It is essential for us that $\cJ$ is a subcollection and not a refinement; by shrinking the intervals we would 
   lose equality in \eqref{assumption2aG}. 
   We will need this proposition for glueing local $L^p$-estimates to global ones.
\end{remark}

We can treat the connected components of $I'$ separately. 
So let $(\al,\be)$ be any connected component of $I'$ and let $\cI := \{J(t_1)\}_{t_1 \in (\al,\be)}$. 
The function $\tilde b := (\tilde b_2,\ldots,\tilde b_{n_b})$ may or may not vanish at the endpoints of $(\al,\be)$.
We distinguish 
three cases:
\begin{enumerate}
  \item[(i)] $\tilde b$ vanishes at both endpoints,
  \begin{equation}
     \label{1case}
     \tilde b(\al) = \tilde b(\be) = 0. 
  \end{equation}
  \item[(ii)] $\tilde b$ vanishes at one endpoint, say $\al$, but not at the other, 
  \begin{equation}
     \label{2case}
     \tilde b(\al) =0, ~ \tilde b(\be) \ne 0.
   \end{equation} 
   \item[(iii)] $\tilde b$ does not vanishes at either endpoint,
   \begin{equation}
    \label{3case}
    \tilde b(\al) \ne 0, ~ \tilde b(\be) \ne 0. 
   \end{equation}
\end{enumerate}

We shall need the following two lemmas.

\begin{lemma} \label{endpoints} 
We have:
\begin{enumerate}
  \item If $\tilde b(\al) = 0$, then no interval $J \in \cI$ has left endpoint $\al$ and $|J(t_1)| \to 0$ as $t_1 \to \al$. 
  If $\tilde b(\be) = 0$, then no interval $J \in \cI$ has right endpoint $\be$ and $|J(t_1)| \to 0$ as $t_1 \to \be$. 
  \item If $\tilde b(\al) \ne 0$, then there exists an interval $J \in \cI$ of second kind (with endpoint $\al$). 
  If $\tilde b(\be) \ne 0$, then there exists an interval $J \in \cI$ of second kind (with endpoint $\be$).
\end{enumerate}
\end{lemma}

\begin{proof}
  (1) By \eqref{b2}, $\tilde b$ is non-zero at both endpoints of $J$. 
  That $|J(t_1)| \to 0$ as $t_1$ tends to an endpoint, where $\tilde b$ vanishes, is immediate from \eqref{assumption2aG}.

  (2) Suppose that $\tilde b(\be) \ne 0$. If all intervals $J(t_1)$ in $\cI$ were of first kind then, 
  by \eqref{assumption2aG} and \eqref{1kind}, 
  \begin{equation} \label{endpoints1}
    \vh_{t_1,+}(\be) \ge \frac{D}{2} |\tilde b_\ell(t_1)|^{1/\ell} 
    = \frac{D}{2} \max_{2 \le i \le n_b} |\tilde b_i(t_1)|^{1/i}, \quad t_1 \in (\al,\be).
  \end{equation}
  But $\vh_{t_1,+}(\be) \to 0$ as $t_1 \to \be$, while the right-hand side of \eqref{endpoints1} tends 
  to a positive constant, a contradiction. 
\end{proof}

\begin{lemma} \label{lem:interlace}
  Let $J \in \cI$ and let $t_1 \not\in J$ be such that $J(t_1)$ is of first kind. 
  Then $J \not \subseteq J(t_1)$.
\end{lemma}

\begin{proof}
  Let $J=J(s_1) = (\al_{s_1},\be_{s_1})$ and assume without loss of generality that $\be_{s_1} \le t_1$. 
  Suppose that $J(s_1) \subseteq J(t_1)$.
  Since $J(t_1)=(\al_{t_1},\be_{t_1})$ is of first kind (cf.\ \eqref{1kind}), we have 
  \begin{align*}
    L (t_1 -\al_{t_1})    
    + \sum_{i=2}^{n_b} \|(\tilde b_i^{1/i})'\|_{L^1 ((\al_{t_1},t_1])} 
    = \vh_{t_1,-}(\al_{t_1})
    =  \frac{D}{2} |\tilde b_{\ell_{t_1}}(t_1)|^{1/\ell_{t_1}} 
    < D |\tilde b_{\ell_{s_1}}(s_1)|^{1/\ell_{s_1}},
  \end{align*}
  because by \eqref{b2} and \eqref{b3} (which follow from \eqref{eq:ell} and \eqref{assumption2aG}),
  \[
    |\tilde b_{\ell_{t_1}}(t_1)|^{1/\ell_{t_1}} < \frac 3 2 |\tilde b_{\ell_{t_1}}(s_1)|^{1/\ell_{t_1}} 
    \le 2 |\tilde b_{\ell_{s_1}}(s_1)|^{1/\ell_{s_1}}. 
  \]
  But this leads to a contradiction in view of \eqref{assumption2aG}.
\end{proof}

Let us now prove Proposition \ref{cover}.

\subsubsection*{Case (i)}
  By \eqref{1case} and Lemma \ref{endpoints}, 
  each $J \in \cI$ is an interval of first kind.

  Choose any interval $J(t_1)$, $t_1 \in (\al,\be)$, and denote it by $J_0 = (\al_0,\be_0)$. 
  Define recursively (for $\ga \in \Z$)
  \begin{align*}
    J_\ga = (\al_\ga,\be_\ga) := 
    \begin{cases}
      J(\be_{\ga-1}) & \text{if } \ga \ge 1, \\
      J(\al_{\ga+1}) & \text{if } \ga \le -1.  
    \end{cases} 
  \end{align*} 
  By Lemma \ref{lem:interlace}, we have $\al < \al_\ga < \al_{\ga+1}$ and $\be_\ga < \be_{\ga+1} < \be$ for all $\ga$.
  Let us show that 
  the collection $\cJ = \{J_\ga\}_{\ga \in \Z}$ covers $(\al,\be)$.   
  Suppose that, say,
  $\ta := \sup_\ga \be_\ga <\be$.  
  By \eqref{assumption2aG} and since all intervals are of first kind (cf.\ \eqref{1kind}),
  \begin{align*}
    L (\ta -\be_{\ga})    
    + \sum_{i=2}^{n_b} \|(\tilde b_i^{1/i})'\|_{L^1 ((\be_{\ga},\ta])} 
    \ge   
    \frac{D}{2} \max_{2 \le i \le n_b} |\tilde b_i(\be_\ga)|^{1/i}. 
  \end{align*}
  But the left-hand side tends to $0$ as $\ga \to +\infty$, whereas the right-hand side converges to 
  $(D/2) \max_{2 \le i \le n_b} |\tilde b_i(\ta)|^{1/i}>0$, a contradiction.

  Now Proposition \ref{cover} follows from Lemma \ref{endpoints} and the following lemma.

  \begin{lemma} \label{glue}
  Let $\cJ = \{J_\ga\}_{\ga \in \Z}$ be a countable collection of bounded open intervals 
  $J_\ga = (\al_\ga,\be_\ga) \subseteq \R$ such that 
  \begin{enumerate}
    \item $\bigcup \cJ = (\al,\be)$ is a bounded open interval,
    \item $\al < \al_\ga < \al_{\ga+1}$ and $\be_\ga < \be_{\ga+1}< \be$ for all $\ga \in \Z$, 
    \item $|J_\ga| \to 0$ as $\ga \to \pm \infty$.
  \end{enumerate}
  Then there is a subcollection $\cJ_0 \subseteq \cJ$ with 
  $\bigcup \cJ_0 = (\al,\be)$ and such that every point in $(\al,\be)$ 
  belongs to at most two intervals in $\cJ_0$. 
\end{lemma}

\begin{proof}
  The assumptions imply that the sequence of left endpoints $(\al_\ga)$ converges to $\be$ as $\ga \to \infty$,
  and the sequence of right endpoints $(\be_\ga)$ converges to $\al$ as $\ga \to -\infty$. 
  Thus, there exists $\ga_1 > 0$ such that $\al_{\ga_1} < \be_0 \le \al_{\ga_1 + 1}$, 
  there exists $\ga_2 > \ga_1$ such that $\al_{\ga_2} < \be_{\ga_1} \le \al_{\ga_2 + 1}$, and iteratively, 
  there exists $\ga_j > \ga_{j-1}$ such that $\al_{\ga_j} < \be_{\ga_{j-1}} \le \al_{\ga_j + 1}$. 
  Symmetrically,  
  there exist integers $\ga_{j-1} < \ga_{j} < 0$ ($j \in \Z_{<0}$) such that 
  $\be_{\ga_{j-1}-1} \le \al_{\ga_j} < \be_{\ga_{j-1}}$.
  Set $\ga_0 := 0$ and define 
  \[
    \cJ_0 := \{J_{\ga_j}\}_{j \in \Z}.
  \]
  By construction $\cJ_0$ still covers $(\al,\be)$ and the left and right endpoints of the intervals $J_{\ga_j}$ are interlacing,
  \[
    \cdots < \be_{\ga_{j-2}}< \al_{\ga_j} < \be_{\ga_{j-1}} < \al_{\ga_{j+1}} < \be_{\ga_{j}} < \al_{\ga_{j+2}} < \cdots 
  \]
  Thus $\cJ_0$ has the required properties.
\end{proof}

Proposition \ref{cover} is proved in Case (i).

\subsubsection*{Case (ii)} By \eqref{2case} and Lemma \ref{endpoints}, the collection $\cI$ contains an interval of second kind. 
  Since $\tilde b(\al) = 0$, all intervals of second kind in $\cI$ must have endpoint $\be$. 
  Thus,   
  \[
    \ta := \inf\{t_1 : J(t_1) \in \cI \text{ is of second kind}\} > \al,
  \] 
  because $|J(t_1)| \to 0$ as $t\to \al$ by Lemma \ref{endpoints}.
  The interval $J(\ta)$ is of first kind (being of second kind is an open condition).  
  There is an interval $J_0 = (\al_0,\be_0=\be)$ of second kind in $\cI$ with $J(\ta) \cap J_0 \ne \emptyset$. 
  Let us denote $J(\ta)$ by $J_{-1} = (\al_{-1},\be_{-1})$ and define recursively  
  \begin{align*}
    J_\ga = (\al_\ga,\be_\ga) := 
      J(\al_{\ga+1}), \quad \ga \le -1. 
  \end{align*} 
  The arguments in Case (i) imply that the collection $\cJ:= \{J_\ga\}_{\ga \le 0}$ is a countable cover of $(\al,\be)$ 
  satisfying $\al < \al_\ga < \al_{\ga+1}$ and $|J_\ga| \to 0$.

  Proposition \ref{cover} follows from (an obvious modification of) Lemma \ref{glue}. This ends Case (ii).

  \subsubsection*{Case (iii)}
  In this case $\cI$ has a finite subcollection $\cJ$ that still covers $(\al,\be)$. Indeed,
  by \eqref{3case} and Lemma \ref{endpoints}, the collection $\cI$ contains 
  intervals of second kind with endpoints $\al$ and $\be$, 
  say, $(\al,\de)$ and $(\ep,\be)$. If their intersection is non-empty we are done. Otherwise there are finitely many   
  intervals in $\cI$ that cover the compact interval $[\de,\ep]$.

  Proposition \ref{cover} follows from the following lemma.  

  \begin{lemma}
    Every finite collection $\cJ$ of open intervals with $\bigcup \cJ = (\al,\be)$ 
    has a subcollection that still covers $(\al,\be)$ and every point in $(\al,\be)$ belongs to at most two 
    intervals in the subcollection.  
  \end{lemma}

  \begin{proof}
    The collection $\cJ$ contains an interval with endpoint $\al$; let $J_0 = (\al = \al_0,\be_0)$ be the biggest among them. 
    If $\be_0 < \be$,
    let $J_1 = (\al_1,\be_1)$ denote the interval among all intervals in $\cJ$ 
    containing $\be_0$ whose right endpoint is maximal.
    If $\be_1 < \be$,
    let $J_2 = (\al_2,\be_2)$ denote the interval among all intervals in $\cJ$ containing $\be_1$ whose right 
    endpoint is maximal, etc. 
    This yields a finite cover of $(\al,\be)$ by intervals $J_i = (\al_i,\be_i)$, $i = 0,1,\ldots,N$, 
    such that $\al_0 < \al_1 < \cdots < \al_N$.
    Define 
    \[
      i_1 := \max_{\al_i <\be_0} i, \qquad i_j := \max_{\al_i <\be_{i_{j-1}}} i,  \quad j \ge 2.  
    \]
    Then $\{J_0,J_{i_1},J_{i_2},\ldots,J_N\}$ has the required properties.
  \end{proof}

  The proof of Proposition \ref{cover} is complete.

\section{Proof of Theorem \ref{main}} \label{proof}

We suppose henceforth that for each integer $n$ a universal splitting of polynomials of degree $n$ in Tschirnhausen form 
in the sense of Section \ref{universal} has been fixed. Whenever we speak of a splitting we mean the 
fixed universal splitting.   
Accordingly, we will apply the following convention:
\begin{quote}
  \it All dependencies of constants on data of the universal splitting, like $\rh$, $\tilde \ps_i$, etc., 
  (see Section \ref{universal}) will no longer 
      be explicitly stated. For simplicity it will henceforth be subsumed by saying that the constants depend on the degree of 
      the polynomials. The constants which are universal in this sense will be denoted by $C$ and may vary from line to line.   
\end{quote}

The heart of the proof of Theorem \ref{main} is the following 
proposition. It comprises the inductive argument on the degree.

\begin{proposition} \label{induction}
  Let $I \subseteq \R$ be a bounded open interval and 
  let $P_{\tilde a}$ be a monic polynomial of degree $n_{\tilde a}$ in Tschirnhausen form with coefficients of class 
  $C^{n_{\tilde a}-1,1}(\oI)$.
  Let $t_0 \in I$ and $k \in \{2,\ldots, n_{\tilde a}\}$ be such that 
  \begin{enumerate}
    \item $|\tilde a_k(t_0)|^{1/k} =  \max_{2 \le j\le n_{\tilde a}} |\tilde a_j(t_0)|^{1/j}\ne 0$,
    \item $\sum_{j=2}^{n_{\tilde a}} \|(\tilde a_j^{1/j})'\|_{L^1(I)} \le B |\tilde a_k(t_0)|^{1/k}$ 
    for some constant $B< 1/3$, 
    \item for all  
    $j = 2, \ldots,n_{\tilde a}$ and 
    $s = 1,\ldots,n_{\tilde a}-1$, 
    \begin{align*}
      \|\tilde a_j^{(s)} \|_{L^\infty(I)} &\le C |I|^{-s}  |\tilde a_k (t_0)|^{j/k},
      \\
      \Lip_I(\tilde a_j^{(n_{\tilde a}-1)} ) &\le C |I|^{-n_{\tilde a}}  |\tilde a_k (t_0)|^{j/k},
    \end{align*}
    where $C= C(n_{\tilde a})$.
    \item Assume that $P_{\tilde a}$ splits on $I$, i.e., $P_{\tilde a}(t) = P_b(t) P_{b^*}(t)$ for $t \in I$, where 
    $b_i$ and $b^*_i$ are given by \eqref{eq:bj}.  
  \end{enumerate}
  Then every continuous root $\mu \in C^0(I)$ of $P_{\tilde b}$ is absolutely continuous and satisfies
  \begin{equation} \label{muab}
    \|\mu'\|_{L^p(I)} \le C  \Big( \| |I|^{-1}  {|\tilde a_k(t_0)|^{1/k}} \|_{L^p (I)} 
    + \sum_{i=2}^{n_b} \|(\tilde b_i^{1/i})'\|_{L^p (I)}\Big),
  \end{equation}   
  for all $1 \le p < (n_{\tilde a})'$ and a constant $C$ depending only on $n_{\tilde a}$ and $p$.  
\end{proposition} 

The proof of Theorem \ref{main} is divided into three steps. 
\begin{description} 
   \item[Step 1] We check that a monic polynomial in Tschirnhausen form satisfying 
   the assumptions of Theorem \ref{main} also satisfies those of Proposition~\ref{induction}.
   \item[Step 2] We prove Proposition \ref{induction}.
   \item[Step 3] We finish the proof of Theorem \ref{main}. 
   The goal is to estimate the right-hand side of \eqref{muab} in terms of the $\tilde a_j$.
 \end{description} 

\subsection*{Step 1: The assumptions of Theorem \ref{main} imply those of Proposition \ref{induction}}

Let $(\al,\be) \subseteq \R$ be a bounded open interval and let
\begin{equation} \label{polynomial}
  P_{\tilde a(t)}(Z) = Z^n + \sum_{j=2}^n \tilde a_j(t) Z^{n-j}, \quad t \in (\al,\be),
\end{equation}
be a monic polynomial in Tschirnhausen form with coefficients $\tilde a_j \in C^{n-1,1}([\al,\be])$, $j = 2,\ldots,n$.

Let $\rh$ be the radius of the fixed universal splitting of polynomials of degree $n$ in Tschirnhausen 
form (cf.\ Section \ref{universal}). 
We fix a universal positive constant $B$ satisfying 
\begin{gather} \label{eq:constB}
  B < \min\Big\{\frac1 3, \frac{\rh}{3 n^2 2^{n}}\Big\}. 
\end{gather}

Fix $t_0 \in (\al,\be)$ and $k \in \{2,\ldots,n\}$ such that \eqref{eq:k} holds, i.e., 
\begin{equation} \label{k}
  |\tilde a_k(t_0)|^{1/k} =  \max_{2 \le j \le n} |\tilde a_j(t_0)|^{1/j}\ne 0
\end{equation}
This is possible unless $\tilde a \equiv 0$ in which case nothing is to prove. 
Choose a maximal open interval $I \subseteq (\al,\be)$ containing $t_0$ such that we have \eqref{assumption}, i.e., 
\begin{align}\label{assumption<=}
\const |I| + \sum_{j=2}^n \|(\tilde a_j^{1/j})'\|_{L^1 (I)} \le  B |\tilde a_k(t_0)|^{1/k},
\end{align}
with $\const$ given by \eqref{M}.
In particular, all conclusions of Section \ref{aestimates} hold true.

Consider the point $\underline p = \ul a(t_0)$, where $\ul a$ is the curve defined in \eqref{curve}. By \eqref{k}, 
$\underline p$ is an element of the set $K$ defined in \eqref{eq:compactK}. 
By the properties of the universal splitting specified in Section \ref{universal},
the ball $B_\rh(\underline p)$ is contained in some ball of the finite cover $\cB$ of $K$. 
By Lemma~\ref{lem2} and \eqref{eq:constB}, 
the length of the curve $\ul a|_I$ is bounded by $\rh$. Thus we have a splitting on $I$, 
\[
  P_{\tilde a}(t) = P_b(t) P_{b^*}(t), \quad t \in I.
\] 
The coefficients $b_i$ of $P_b$ are given by \eqref{eq:b_i0}, and, after the Tschirnhausen transformation 
$P_b \leadsto P_{\tilde b}$, 
the coefficients $\tilde b_i$ of $P_{\tilde b}$ are given by \eqref{eq:b_i}.  
(Similar formulas hold for $b_i^*$ and $\tilde b_i^*$.)

In summary, the restriction of the curve of polynomials $P_{\tilde a}$ 
to the interval $I$ satisfies all assumptions and thus all conclusions of  
Sections \ref{aestimates} and \ref{bestimates}. 
In particular, the assumptions of Proposition \ref{induction}
are satisfied. Thus we have proved the following lemma.

\begin{lemma} \label{lem:assThm1implyassProp3}
  Let $(\al,\be) \subseteq \R$ be a bounded open interval and let $P_{\tilde a}$ be a polynomial \eqref{polynomial}
  in Tschirnhausen form with coefficients $\tilde a_j \in C^{n-1,1}([\al,\be])$, $j = 2,\ldots,n$.  
  Let $B$ be a positive constant satisfying \eqref{eq:constB}. 
  Let $t_0 \in (\al,\be)$ and $k \in \{2,\ldots,n\}$ be such that \eqref{k} holds.
  Let $I$ be an open interval with $t_0 \in I \subseteq (\al,\be)$ and satisfying \eqref{assumption<=}. 
  Then the assumptions \thetag{1}--\thetag{4} of Proposition \ref{induction} are fulfilled. 
\end{lemma}

\subsection*{Step 2: Induction on the degree}

Let us prove Proposition \ref{induction}.

We proceed by induction on the degree $n = n_{\tilde a}$.  
The assumptions of the proposition amount exactly to the assumptions \eqref{eq:polynomial}--\eqref{assumpt}, \eqref{est:a}, and 
\eqref{splitting}--\eqref{eq:b_i}. Thus we may rely on all conclusions of Sections \ref{aestimates} and \ref{bestimates}.

\subsubsection*{Induction basis} 

Proposition \ref{induction} trivially holds for 
polynomials of degree $1$. 
Using the result of Ghisi and Gobbino, i.e., Corollary \ref{cor:radicals}, 
one can also check that Proposition~\ref{induction} is valid 
for polynomials of the form $P_{\tilde a}(Z) = Z^n - \tilde a_n$, $n\ge 2$, because they 
can be split into the product of linear factors 
$P_{\tilde a}(Z) = \prod_{\xi^n=1}  (Z -  \xi \tilde a_n^{1/n})$. 
But we do not need 
to consider this case separately, since it will appear implicitly in the inductive step.

\subsubsection*{Inductive step} 

By \eqref{eq:ass11}, $\tilde a_k$ does not vanish on $I$, and thus $b_i$ and $\tilde b_i$ belong to $C^{n-1,1}(\oI)$.
Let us set 
\[
  I' := I \setminus \{t \in I : \tilde b_2 (t) = \cdots = \tilde b_{n_b}(t) = 0\}.
\]
For each $t_1 \in I'$ choose $\ell \in \{2, \ldots, n_b\}$ such that \eqref{eq:ell} holds. 
By Section \ref{subintervals}, there is an open interval $J = J(t_1)$, $t_1 \in J \subseteq I'$, 
such that \eqref{assumption2a}.
The constant $D$ in \eqref{assumption2a} 
can be chosen sufficiently small such that on $J$ we have a splitting 
\begin{equation*}
  P_{\tilde b}(t) = P_c(t) P_{c^*}(t), \quad t \in J;
\end{equation*}
in fact, it suffices to choose 
\begin{equation} \label{D}
  D < \min\Big\{\frac 1 3, \frac{\si}{3 n_b^2 2^{n_b}}, C_1^{-1}\Big\}, 
\end{equation}
where $C_1$ is the constant in \eqref{eq:b_ibound} and where $\si$ is the radius of the universal splitting of polynomials 
of degree $n_b$ in Tschirnhausen form.
Indeed, the length of the curve $\underline b|_J$ is bounded by $\si$, which follows 
from Lemma~\ref{lemB}, 
and the arguments in Section \ref{universal} and in Step~1 applied to $P_{\tilde b}$.

By Proposition \ref{cover} (where \eqref{assumption2a} plays the role of \eqref{assumption2aG}), 
we may conclude that there is a countable family $\{(J_\ga,t_\ga,\ell_\ga)\}$ of open intervals $J_\ga \subseteq I'$, 
of points $t_\ga \in J_\ga$, 
and of integers $\ell_\ga \in \{2, \ldots, n_b\}$ 
satisfying
\begin{gather}
    |\tilde b_{\ell_\ga}(t_\ga)|^{1/{\ell_\ga}} = \max_{2 \le i \le n_b} |\tilde b_i(t_\ga)|^{1/i} \ne 0, \\
    |J_\ga|  |I|^{-1}  {|\tilde a_k(t_0)|^{1/k}}  + \sum_{i=2}^{n_b} \|(\tilde b_i^{1/i})'\|_{L^1 (J_\ga)} =  
    D |\tilde b_{\ell_\ga}(t_\ga)|^{1/\ell_\ga}, \label{Jba}\\ 
    P_{\tilde b}(t) = P_{c_\ga}(t) P_{c_\ga^*}(t), \quad t \in J_\ga, \label{Pb}\\
    \bigcup_\ga J_\ga = I', \quad \sum_\ga |J_\ga| \le 2 |I'|. \label{Jga}
\end{gather}
In particular, for every $\ga$, the polynomial $P_{\tilde b}(t) = P_{c_\ga}(t) P_{c_\ga^*}(t)$, $t \in J_\ga$, satisfies the 
assumptions of Proposition \ref{induction}; note that (3) in Proposition \ref{induction} corresponds to \eqref{b4}.

Let $\mu \in C^0(I)$ be a continuous root of $P_{\tilde b}$. 
We may assume without loss of generality that in $J_\ga$, 
\begin{equation} \label{tildemu}
  \tilde \mu(t) := \mu(t) + \frac{c_{\ga 1}(t)}{n_{c_\ga}}, \quad t \in J_\ga,
\end{equation}
is a root of $P_{\tilde c_\ga}$, where $n_{c_\ga} := \deg P_{c_\ga}$. 
Since $n_{c_\ga} < n_{b} < n_{\tilde a}$, the induction hypothesis implies that
$\tilde \mu$ is absolutely continuous and satisfies 
\begin{equation} \label{mu}
    \|\tilde \mu'\|_{L^p(J_\ga)} \le C \Big( \| |J_\ga|^{-1}  {|\tilde b_{\ell_\ga}(t_\ga)|^{1/\ell_\ga}} \|_{L^p (J_\ga)} 
  + \sum_{h=2}^{n_{c_\ga}} \|(\tilde c_{\ga h}^{1/h})'\|_{L^p (J_\ga)}\Big),
\end{equation}
for all $1 \le p < (n_b)'$, for a constant $C$ depending only on $n_b$ and $p$.

\subsubsection*{\texorpdfstring{$L^p$}{Lp}-estimates on \texorpdfstring{$I$}{I}}

To finish the proof of Proposition \ref{induction} we have to show that the estimates \eqref{mu} 
on the subintervals $J_\ga$
imply the bound \eqref{muab} on $I$. 
To this end we claim that, 
for all $p$ with $1 \le p < (n_{c_{\ga}})'$, 
\begin{equation} \label{eq:ctob}
  \sum_{h=2}^{n_{c_\ga}} \|(\tilde c_{\ga h}^{1/h})'\|^*_{L^p(J_\ga)} 
  \le C  |J_\ga|^{-1} |\tilde b_{\ell_\ga} (t_\ga)|^{1/\ell_\ga}, 
\end{equation}
for a constant $C$ that depends only on $n_{\tilde a}$ and $p$.

By the properties of the universal splitting (cf.\ Sections \ref{ssec:split} and \ref{universal}), 
the coefficients $c_{\ga h}$ of $P_{c_\ga}$ are of the form 
\[
  c_{\ga h} = \tilde b_{\ell_\ga}^{h/\ell_\ga} 
  \th_h \big(\tilde b_{\ell_\ga}^{-2/\ell_\ga} \tilde b_2, \ldots, 
  \tilde b_{\ell_\ga}^{-n_b/\ell_\ga} \tilde b_{n_b}\big), 
  \quad h = 1,\ldots, n_{c_\ga},
\]  
and after the Tschirnhausen transformation $P_{c_{\ga}} \leadsto P_{\tilde c_\ga}$, see \eqref{eq:tildebj}, 
\[
  \tilde c_{\ga h} = \tilde b_{\ell_\ga}^{h/\ell_\ga} 
  \tilde \th_h \big(\tilde b_{\ell_\ga}^{-2/\ell_\ga} \tilde b_2, \ldots, 
  \tilde b_{\ell_\ga}^{-n_b/\ell_\ga} \tilde b_{n_b}\big), 
  \quad h = 2,\ldots, n_{c_\ga},
\] 
where $\th_h$, respectively, $\tilde \th_h$, are analytic functions with bounded partial derivatives of all orders. 
By \eqref{b2}, $\tilde b_{\ell_\ga}$ does not vanish on $J_\ga$ and thus $c_{\ga h}$ and $\tilde c_{\ga h}$ belong to 
$C^{n_{\tilde a}-1,1}(\overline J_\ga)$.
By Lemma \ref{lem:B} (applied to $\tilde c_{\ga h}$, $J_\ga$, $|\tilde b_{\ell_\ga} (t_\ga)|^{1/\ell_\ga}$ instead of 
$\tilde b_i$, $I$, $|\tilde a_k(t_0|^{1/k}$), 
we find that, for $h = 2,\ldots, n_{c_\ga}$ and $s=1,\dots,n_{\tilde a}-1$,
\begin{align*}
  \|\tilde c_{\ga h}^{(s)} \|_{L^\infty(J_\ga)} &\le C  |J_\ga|^{-s}  |\tilde b_{\ell_\ga} (t_\ga)|^{h/\ell_\ga},
  \\
  \Lip_{J_\ga} (\tilde c_{\ga h}^{(n_{\tilde a}-1)} ) 
  &\le C  |J_\ga|^{-n_{\tilde a}}  |\tilde b_{\ell_\ga} (t_\ga)|^{h/\ell_\ga},    
\end{align*}    
where $C=C(n_{\tilde a})$.
Then Lemma \ref{lem:Lpbak} yields \eqref{eq:ctob}.

Now \eqref{eq:ctob}, \eqref{Jba}, and \eqref{inclusions0} allow us to estimate the right-hand side of 
\eqref{mu}:
\begin{align*}
\MoveEqLeft
  \||J_\ga|^{-1} |\tilde b_{\ell_\ga}(t_\ga)|^{1/\ell_\ga} \|^*_{L^p (J_\ga)}
  + \sum_{h=2}^{n_{c_\ga}} \|(\tilde c_{\ga h}^{1/h})'\|^*_{L^p (J_\ga)}
  \\
  &\le (1+C) |J_\ga|^{-1} |\tilde b_{\ell_\ga}(t_\ga)|^{1/\ell_\ga}
  \\
  &= (1+C) D^{-1} \Big( \| |I|^{-1}  {|\tilde a_k(t_0)|^{1/k}} \|^*_{L^1 (J_\ga)} 
  + \sum_{i=2}^{n_b} \|(\tilde b_i^{1/i})'\|^*_{L^1 (J_\ga)}\Big)
  \\
  &\le (1+C)   D^{-1}   \Big( \| |I|^{-1}  {|\tilde a_k(t_0)|^{1/k}} \|^*_{L^p (J_\ga)} 
  + \sum_{i=2}^{n_b} \|(\tilde b_i^{1/i})'\|^*_{L^p (J_\ga)}\Big) 
\end{align*}
and therefore 
\begin{align}
\MoveEqLeft
  \||J_\ga|^{-1} |\tilde b_{\ell_\ga}(t_\ga)|^{1/\ell_\ga} \|^p_{L^p (J_\ga)}
  + \sum_{h=2}^{n_{c_\ga}} \|(\tilde c_{\ga h}^{1/h})'\|^p_{L^p (J_\ga)}
  \notag \\
  &\le C D^{-p}  \Big( \| |I|^{-1}  {|\tilde a_k(t_0)|^{1/k}} \|^p_{L^p (J_\ga)} 
  + \sum_{i=2}^{n_b} \|(\tilde b_i^{1/i})'\|^p_{L^p (J_\ga)}\Big), \label{cba}
\end{align}
for a constant $C$ that depends only on $n_{\tilde a}$ and $p$.

By Remark \ref{rem:b1prime} (applied to $\tilde c_{\ga h}$, $J_\ga$, $|\tilde b_{\ell_\ga} (t_\ga)|^{1/\ell_\ga}$ instead of 
$\tilde b_i$, $I$, $|\tilde a_k(t_0|^{1/k}$), we have
\begin{align*}  
    \|c_{\ga 1}'\|_{L^\infty(J_\ga)}   \le C  |J_\ga|^{-1}  |\tilde b_{\ell_\ga} (t_\ga)|^{1/\ell_\ga},
\end{align*} 
where $C=C(n_{\tilde a})$.
Thus, using \eqref{Jba} and \eqref{inclusions0}, we find (as in the derivation of \eqref{cba})
\begin{align}  \label{c1prime}
    \|c_{\ga 1}'\|^p_{L^p(J_\ga)} \le C D^{-p}  \Big( \| |I|^{-1}  {|\tilde a_k(t_0)|^{1/k}} \|^p_{L^p (J_\ga)} 
  + \sum_{i=2}^{n_b} \|(\tilde b_i^{1/i})'\|^p_{L^p (J_\ga)}\Big),
\end{align}
for a constant $C$ that depends only on $n_{\tilde a}$ and $p$.

Let us now glue the bounds on $J_\ga$ to a bound on $I$.
By \eqref{Jga}, \eqref{mu}, \eqref{cba}, and \eqref{c1prime},
\begin{align} \label{ind11}
  \sum_\ga \|\tilde \mu'\|^p_{L^p(J_\ga)} 
  &\le  C D^{-p}  \Big( \| |I|^{-1}  {|\tilde a_k(t_0)|^{1/k}} \|^p_{L^p (I)} 
  + \sum_{i=2}^{n_b} \|(\tilde b_i^{1/i})'\|^p_{L^p (I)}\Big), 
\end{align}
and
\begin{align} \label{ind22}
  \sum_\ga \|c_{\ga 1}'\|^p_{L^p(J_\ga)} 
  \le C D^{-p}  \Big( \| |I|^{-1}  {|\tilde a_k(t_0)|^{1/k}} \|^p_{L^p (I)} 
  + \sum_{i=2}^{n_b} \|(\tilde b_i^{1/i})'\|^p_{L^p (I)}\Big), 
\end{align}
for a constant $C$ that depends only on $n_{\tilde a}$ and $p$.
By \eqref{Jga}, \eqref{tildemu}, \eqref{ind11}, and \eqref{ind22}, 
we may conclude that $\mu$ is absolutely continuous on $I'$ 
and 
\begin{align*}
  \|\mu'\|_{L^p(I')} 
  \le C D^{-1}  \Big( \| |I|^{-1}  {|\tilde a_k(t_0)|^{1/k}} \|_{L^p (I)} 
  + \sum_{i=2}^{n_b} \|(\tilde b_i^{1/i})'\|_{L^p (I)}\Big),
\end{align*}
for a constant $C$ that depends only on $n_{\tilde a}$ and $p$.
Since $\mu$ vanishes on $I \setminus I'$, Lemma~\ref{lem:extend} implies that $\mu$ is absolutely continuous on $I$ and  
satisfies \eqref{muab}, since $D = D(n_{\tilde a})$ by \eqref{D}. 
This completes the proof of Proposition \ref{induction}.

\subsection*{Step 3: End of the proof of Theorem \ref{main}}

We have seen in Lemma \ref{lem:assThm1implyassProp3} that for a polynomial $P_{\tilde a}$ in Tschirnhausen form 
\eqref{polynomial}
satisfying \eqref{k} and \eqref{assumption<=}
the assumptions of Proposition~\ref{induction} hold with the constant $B$ 
fulfilling \eqref{eq:constB}. 
Our next goal is to  estimate the right-hand side of \eqref{muab} in terms of the $\tilde a_j$.

By Lemma \ref{lem:B}, we have \eqref{eq:b_ider}, and thus, by Lemma \ref{lem:Lpbak},
we get for all $p$ with $1 \le p < (n_b)'$, 
\begin{align} \label{beforecases}
  \||I|^{-1} |\tilde a_k (t_0)|^{1/k} \|^*_{L^p (I)} + \sum_{i=2}^{n_b} \|(\tilde b_i^{1/i})'\|^*_{L^p(I)} 
  \le C |I|^{-1} |\tilde a_k (t_0)|^{1/k}
\end{align}
where the constant $C$ depends only on $n$ and $p$.

At this stage two cases may occur:
\begin{enumerate}
  \item[(i)] Either we have equality in \eqref{assumption<=}, i.e., 
  \begin{align}\label{assumption=}
    \const |I| + \sum_{j=2}^n \|(\tilde a_j^{1/j})'\|_{L^1 (I)} =  B |\tilde a_k(t_0)|^{1/k} .
  \end{align}
  \item[(ii)] Or $I = (\al,\be)$ and 
  \begin{align}\label{assumption<}
    \const |I| + \sum_{j=2}^n \|(\tilde a_j^{1/j})'\|_{L^1 (I)} <  B |\tilde a_k(t_0)|^{1/k} .
  \end{align}
\end{enumerate}
Case (ii) entails an unpleasant blow-up of the bounds if $\be - \al \to 0$ as explained in the following remark. 
We will explain below how to avoid this phenomenon.

\begin{remark} \label{rem:blowup}
In Case (ii) we have a splitting  $P_{\tilde a} = P_b P_{b^*}$ on the whole interval $I=(\al,\be)$; cf.\ 
Step 1. Thus, \eqref{beforecases} becomes 
\begin{align*}
  \||(\be-\al)^{-1} |\tilde a_k (t_0)|^{1/k} \|_{L^p ((\al,\be))} 
  + \sum_{i=2}^{n_b} \|(\tilde b_i^{1/i})'\|_{L^p((\al,\be))} 
  &\le C  (\be-\al)^{-1+1/p} |\tilde a_k(t_0)|^{1/k}
\end{align*}
which can be bounded by 
\begin{align} \label{eq:blowup}
  C  (\be-\al)^{-1+1/p} \max_{2 \le j \le n} \|\tilde a_j\|^{1/j}_{L^\infty((\al,\be))}.
\end{align}
Similarly, \eqref{b1prime} implies that $\|b_1'\|_{L^p((\al,\be))}$ is bounded by \eqref{eq:blowup}.
If $\la \in C^0((\al,\be))$ is a continuous root of $P_{\tilde a}$ then we may assume that it is a root of $P_b$, 
and hence $\la = \mu - b_1/n_b$, for a continuous root $\mu \in C^0((\al,\be))$ of $P_{\tilde b}$.  
By \eqref{muab}, we may conclude that $\la$ is absolutely continuous on $(\al,\be)$ and 
\begin{align} \label{lacase2}
  \|\la'\|_{L^p((\al,\be))} \le C (\be-\al)^{-1+1/p} \max_{2 \le j \le n} \|\tilde a_j\|^{1/j}_{L^\infty((\al,\be))},
\end{align}
where $C = C(n,p)$. 
But the bound for $\|\la'\|_{L^p((\al,\be))}$ in \eqref{lacase2} tends to infinity if $\be - \al \to 0$ unless $p=1$. 
\end{remark}

The next lemma provides a way to enforce Case (i). 

\begin{lemma} \label{lem:whitney}
  Let $-\infty<\al<\be<\infty$.
  Let $\tilde a_j \in C^{n-1,1}([\al,\be])$, for $2,\ldots,n$. Let $\hat\al := \al-1$ and $\hat \be := \be +1$.
  The functions $\tilde a_j$ can be extended 
  to functions, again denoted by $\tilde a_j$, 
  defined on $(\hat\al,\hat \be)$ such that the following holds. We have 
  \begin{equation} \label{whitney}
    \|\tilde a_j\|_{C^{n-1,1}([\hat\al,\hat \be])} \le C\, \|\tilde a_j\|_{C^{n-1,1}([\al,\be])},
\end{equation}
for some universal constant $C$ independent of $(\al,\be)$. 
For each $t_0 \in (\hat\al,\hat \be)$ and $k \in \{2,\ldots,n\}$ satisfying \eqref{k} 
  there is an open interval $I \subseteq (\hat\al,\hat \be)$ containing $t_0$ such that \eqref{assumption=} holds true 
  with $B$ specified in \eqref{eq:constB} and $M$ defined in \eqref{M}.
\end{lemma}

\begin{proof}
Using a simple version of Whitney's extension theorem (cf.\ \cite[Theorem 4, p.177]{Stein70}), we may extend the functions 
$\tilde a_j \in C^{n-1,1}([\al,\be])$ to functions, again denoted by $\tilde a_j$, 
defined on $\R$ such that $\tilde a_j, \tilde a_j',\ldots, \tilde a_j^{(n-1)}$ are continuous and bounded on $\R$ and 
$\Lip_\R (\tilde a_j^{(n-1)}) < \infty$. More precisely, 
\begin{equation*} 
  \max_{0 \le i \le n-1} \|\tilde a_j^{(i)}\|_{L^\infty(\R)} + \Lip_{\R}(\tilde a_j^{(n-1)}) 
  \le C\, \|\tilde a_j\|_{C^{n-1,1}([\al,\be])},
\end{equation*}
for some universal constant $C$ independent of $(\al,\be)$. Choose a smooth function $\vh : \R \to [0,1]$ such that 
$\vh(t) = 1$ for $t \le 0$ and $\vh(t) = 0$ for $t \ge 1$. Then 
\[
  \ps(t) := \vh(\al - t) \vh(t-\be), \quad t \in \R,
\]
is a smooth function which is $1$ on the interval $[\al,\be]$ and $0$ outside the interval $[\hat \al, \hat \be]$.  
By multiplying all functions $\tilde a_j$ with the cut-off function $\ps$, we may assume that each 
$\tilde a_j$ vanishes somewhere in $[\hat\al,\hat \be]$ and the Leibniz formula implies \eqref{whitney} 
for a constant $C$ depending only on $\vh$.

If there is a point $s = s(j) \in [\hat\al,\hat \be]$ such that $\tilde a_j(s) =0$, then, for $t \in [\hat\al,\hat \be]$, 
  \[
    |\tilde a_j^{1/j}(t)| = |\int_s^t (\tilde a_j^{1/j})'\, d \ta| \le \|(\tilde a_j^{1/j})'\|_{L^1((\hat\al,\hat \be))}
  \]
  and hence 
  \begin{equation} \label{sufficient}
      \max_{2 \le j \le n} \|\tilde a_j\|_{L^\infty((\hat\al,\hat \be))}^{1/j} 
      \le \sum_{j=2}^n \|(\tilde a_j^{1/j})'\|_{L^1((\hat\al,\hat \be))}. 
  \end{equation}
  Since $B<1$ (by \eqref{eq:constB}), \eqref{sufficient} enforces Case (i): 
  for each $t_0 \in (\hat\al,\hat \be)$ and $k \in \{2,\ldots,n\}$ satisfying \eqref{k} 
  there is an open interval $I \subseteq (\hat\al,\hat \be)$ containing $t_0$ such that \eqref{assumption=} holds true.   
\end{proof}

\begin{lemma} \label{lem:local}
  Let $P_{\tilde a}$ be a monic polynomial of degree $n$ in Tschirnhausen form \eqref{polynomial}  with coefficients of class 
  $C^{n-1,1}([\hat \al,\hat \be])$. 
  Let $t_0 \in (\hat \al,\hat \be)$, $k \in \{2,\cdots, n\}$, and let $I \subseteq (\hat \al,\hat \be)$ be an 
  open interval containing $t_0$ such that \eqref{k} and \eqref{assumption=} hold with 
  the constant $B$ 
  fulfilling \eqref{eq:constB} and $\const$ defined by \eqref{M}. 
  Then any continuous root $\la \in C^0(I)$ of $P_{\tilde a}$ on $I$ is absolutely continuous on $I$ and satisfies 
  \begin{align} \label{laIA}
  \|\la'\|_{L^p(I)} \le C   \Big(\hat A \|1\|_{L^p(I)} + \sum_{j=2}^n \|(\tilde a_j^{1/j})'\|_{L^p (I)}\Big),
  \end{align}
  where
  \begin{equation} \label{tildeA}
   \hat A := \max_{2 \le j \le n} \|\tilde a_j\|^{1/j}_{C^{n-1,1}([\hat\al,\hat \be])}.
\end{equation}
\end{lemma}

\begin{proof}
By Lemma \ref{lem:assThm1implyassProp3} (for $(\hat \al,\hat \be)$ instead of $(\al,\be)$), 
the assumptions of Proposition \ref{induction} are satisfied.
In particular, we have a splitting $P_{\tilde a}(t) = P_b(t) P_{b^*}(t)$ for $t \in I$. 
We may assume without loss of generality that $\la$  
is a root of $P_b$.
Then it has the form 
\begin{align} \label{lambda}
  \la(t) &= - \frac{b_1(t)}{n_b} + \mu(t), \quad t \in I, 
\end{align}
where $\mu$ is a continuous root of $P_{\tilde b}$. 
By Proposition \ref{induction}, $\mu$ is absolutely continuous on $I$ and satisfies \eqref{muab}.

  Using \eqref{assumption=} and \eqref{inclusions0} to estimate \eqref{beforecases} (as in the derivation of \eqref{cba}), 
we arrive at 
\begin{align}
  \MoveEqLeft
  \||I|^{-1} |a_k (t_0)|^{1/k} \|_{L^p (I)} + \sum_{i=2}^{n_b} \|(\tilde b_i^{1/i})'\|_{L^p(I)} 
  \le C    \Big(\const\|1\|_{L^p(I)} + \sum_{j=2}^n \|(\tilde a_j^{1/j})'\|_{L^p (I)}\Big), \label{Lpestimate} 
\end{align}
for a constant $C$ that depends only on $n$ and $p$; note that $B=B(n)$ by \eqref{eq:constB}.
Thus, by \eqref{muab} and \eqref{Lpestimate}, 
\begin{equation*} 
  \|\mu'\|_{L^p(I)} \le C  \Big(\const \|1\|_{L^p(I)} + \sum_{j=2}^n \|(\tilde a_j^{1/j})'\|_{L^p (I)}\Big).
\end{equation*}
By \eqref{inclusions0}, \eqref{b1prime}, \eqref{assumption=}, and \eqref{Lpestimate}, we have the same bound for 
$\|b_1'\|_{L^p(I)}$,
and, in view of \eqref{lambda}, we conclude that $\la$ is absolutely continuous on $I$ and satisfies 
\begin{align*} 
  \|\la'\|_{L^p(I)} \le C   \Big(\const \|1\|_{L^p(I)} + \sum_{j=2}^n \|(\tilde a_j^{1/j})'\|_{L^p (I)}\Big).
\end{align*}
The constant $\const$, defined in \eqref{M}, which depends on $t_0$ and $I$ can be bounded by $\hat A$ defined in 
\eqref{tildeA};
in fact,
\begin{align*}
  \const &= \max_{2 \le j\le n}  (\Lip_I(\tilde a_j^{(n-1)}))^{1/n} |\tilde a_k (t_0)|^{(n-j)/(kn)}
  \le \max_{2 \le j\le n} \hat A^{j/n} \hat A^{(n-j)/n} = \hat A.  
\end{align*}
This entails \eqref{laIA}.   
\end{proof}

\begin{proof}[Proof of Theorem \ref{main}]
Let $(\al,\be) \subseteq \R$ be a bounded open interval and let
\begin{equation} \label{polynomialbeforeTschirnhausen}
  P_{a(t)}(Z) = Z^n + \sum_{j=1}^n a_j(t) Z^{n-j}, \quad t \in (\al,\be),
\end{equation}
be a monic polynomial with coefficients $a_j \in C^{n-1,1}([\al,\be])$, $j = 1,\ldots,n$.

Without loss of generality we may assume that $n \ge 2$ and that $P_a = P_{\tilde a}$ is in Tschirnhausen form, 
i.e., $\tilde a_1=0$.
We shall see at the end of the proof how to get the bound \eqref{bound} from a corresponding bound involving the $\tilde a_j$. 
If $\{\la_j(t)\}_{j=1}^n$, $t \in (\al,\be)$, is any system of the roots of $P_{\tilde a}$ (not necessarily continuous), 
then, since $\tilde a_1=0$, for fixed $t \in (\al,\be)$,
\begin{equation} \label{zero}
  \forall_{i,j} ~\la_i(t) = \la_j(t) \quad \Longleftrightarrow \quad 
  \forall_i ~\la_i(t) =0 \quad \Longleftrightarrow \quad 
  \forall_i ~\tilde a_i(t) =0.
\end{equation}

  Let $\la \in C^0((\al,\be))$ be a continuous root of $P_{\tilde a}$. 
  We use Lemma \ref{lem:whitney} to extend $P_{\tilde a}$ to the interval $[\hat \al , \hat \be]$. 
  We extend $\la$ continuously to the interval $(\hat \al,\hat \be)$
  such that $P_{\tilde a(t)}(\la(t)) = 0$ for all $t \in (\hat \al,\hat \be)$.
  Then, by Lemma \ref{lem:local} and Proposition \ref{cover} (applied to $\tilde a_j$ instead of $\tilde b_i$ and 
  \eqref{assumption=} instead of \eqref{assumption2aG}), 
  we can cover the 
  complement in $(\hat\al,\hat \be)$ of the points $t$ satisfying \eqref{zero} 
by a countable family $\cI$ of open intervals $I$ on which \eqref{laIA} holds and such that 
$\sum_{I \in \cI} |I| \le 2 (\hat \be-\hat\al)$.
Since $\la$ vanishes on the points $t$ satisfying \eqref{zero}, Lemma~\ref{lem:extend} yields that 
$\la$ is absolutely continuous on $(\hat\al,\hat \be)$ and satisfies
\begin{equation*}
  \|\la'\|_{L^p((\hat\al,\hat \be))} \le  C   \Big(\hat A \|1\|_{L^p((\hat\al,\hat \be))} 
  + \sum_{j=2}^n \|(\tilde a_j^{1/j})'\|_{L^p ((\hat\al,\hat \be))}\Big),
\end{equation*}
and, using  \eqref{est}, we obtain   
\begin{equation*} 
  \|\la'\|_{L^p((\hat\al,\hat \be))} \le  C   \Big(\hat A (\hat \be-\hat\al)^{1/p} 
  + \sum_{j=2}^n \max \Big\{ (\Lip_{(\hat\al,\hat \be)}(\tilde a_j^{(j-1)}))^{1/j}  (\hat \be-\hat\al)^{1-1/j}, 
  \|\tilde a_j'\|^{1/j}_{L^\infty((\hat\al,\hat \be))}
   \Big\} \Big),
\end{equation*}
where $C = C(n,p)$.

Now let us restrict to the interval $(\al,\be)$ again, and set 
\begin{equation*} 
   \tilde A := \max_{2 \le j \le n} \|\tilde a_j\|^{1/j}_{C^{n-1,1}([\al, \be])}.
\end{equation*} 
By \eqref{whitney} and \eqref{tildeA}, we have
$\hat A \le   C \, \tilde A$
for a universal constant $C$. 
Moreover, 
$\hat \be -\hat \al = \be -\al + 2$ and $1 - 1/j < 1/p$ for all $j \le n$. 
Consequently, 
\begin{align} \label{lacase1}
  \|\la'\|_{L^p((\al, \be))} \le  C(n,p) \max\{1, (\be-\al)^{1/p} \}  \tilde A. 
\end{align}

Finally we determine the bound in terms of the $a_j$ (i.e., \emph{before} the Tschirnhausen transformation).
Let $\check \la := \la - a_1/n$, i.e., $\check \la$ is a continuous root of $P_a$, and set
\[
  A := \max_{1 \le j \le n} \|a_j\|^{1/j}_{C^{n-1,1}([\al,\be])}. 
\]
Thanks to the weighted homogeneity of the formulas \eqref{Tschirnhausen}, $\tilde A \le C(n) A$.
Thus, by \eqref{lacase1}, 
\begin{align*}
  \|\check \la'\|_{L^p((\al,\be))} 
  &\le \|\la'\|_{L^p((\al,\be))} +  \|a_1'\|_{L^p((\al,\be))}
  \\
  &\le C(n,p) \max\{1, (\be-\al)^{1/p} \}  \tilde A + (\be-\al)^{1/p} \|a_1'\|_{L^\infty((\al,\be))}  
  \\
  &\le C(n,p) \max\{1, (\be-\al)^{1/p}\}   A,
\end{align*}
that is \eqref{bound}.
The proof of Theorem \ref{main} is complete. 
\end{proof}

\section{Proof of Theorem \ref{main2}} \label{proofThm2}

Theorem \ref{main2} follows from Theorem \ref{main} 
by the arguments given in the proof of \cite[Theorem 4.1]{ParusinskiRainerAC}. We provide full details in order to 
see that the constant in the bound \eqref{multbound} depends only on the cover $\cK$ of $\overline V$ 
(apart from $m$, $n$, and $p$); 
this will be important in forthcoming work.

\begin{proof}[Proof of Theorem \ref{main2}]
  By Theorem~\ref{main}, $\la$ is absolutely continuous along affine lines parallel to 
  the coordinate axes (restricted to $V$). 
  So $\la$ possesses the partial derivatives $\p_i \la$, $i=1,\ldots,m$, which 
  are defined almost everywhere and are measurable. 
  
  Set $x=(t,y)$, where $t=x_1$, $y=(x_2,\ldots,x_m)$, 
  and let $V_1$ be the orthogonal projection of $V$ on the hyperplane $\{x_1=0\}$.
  For each $y \in V_1$ we denote by $V^y := \{t \in \R : (t,y) \in V\}$ the corresponding section of $V$; 
  note that $V^y$ is open in $\R$. 
  
  We may cover $\overline V$ by finitely many open boxes $K = I_1 \times \cdots \times I_m$ contained in $U$.
  Let $K$ be fixed and set $L = I_2 \times \cdots \times I_m$. 
  Fix $y \in V_1 \cap L$ and let $\la^y_j$, $j=1,\ldots,n$, be a continuous system of the roots of $P_a(~,y)$ on 
  $\Om^y := V^y \cap I_1$ such that $\la(~,y) = \la^y_1$; it exists 
  since $\la(~,y)$ can be completed to a continuous system of the roots of $P_a(~,y)$ on each connected component 
  of $\Om^y$ by  \cite[Lemma 6.17]{RainerN}. 
  Our goal is to bound  
  \[
    \|\p_t\la(~,y)\|_{L^p(\Om^y)} = \|(\la^y_1)'\|_{L^p(\Om^y)}
  \]
  uniformly with respect to $y \in V_1 \cap L$.

  \begin{figure}[h]
    \includegraphics[scale=0.5]{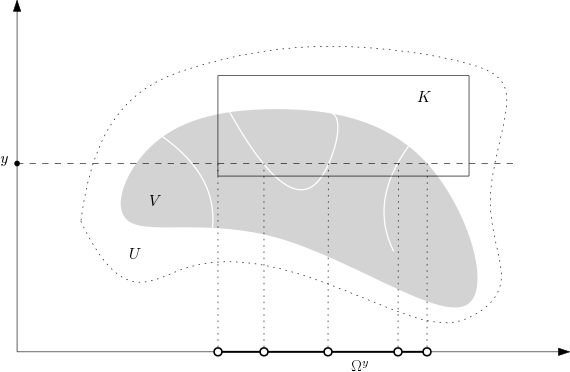}
  \end{figure}

  To this end let $\cC^y$ denote the set of connected components $J$ of the open subset $\Om^y \subseteq \R$. 
  For each $J \in \cC^y$ we extend the system of roots $\la^y_j|_J$, $j=1,\ldots,n$, continuously to $I_1$, i.e., 
  we choose continuous functions $\la^{y,J}_j$, $j = 1,\ldots, n$, on $I_1$ such that $\la^{y,J}_j|_J = \la^y_j|_J$ 
  for all $j$ and 
  \begin{gather*}
      P_a(t,y)(Z) = \prod_{j=1}^n (Z-\la^{y,J}_j(t)), ~t \in I_1.
  \end{gather*} 
  This is possible since $\la^y_j|_J$ has a continuous extension to the endpoints of the (bounded) interval $J$, by 
  \cite[Lemma 4.3]{KLMR05}, and can then be extended on the left and on the right of $J$ by a continuous system 
  of the roots of $P_a(~,y)$ after suitable permutations.

  By Theorem~\ref{main}, for each $y \in V_1\cap L$, $J \in \cC^y$, and $j=1,\ldots,n$, 
  the function $\la^{y,J}_j$ is absolutely continuous on $I_1$ and $(\la^{y,J}_j)' \in L^p(I_1)$, for $1 \le p < n/(n-1)$, with 
  \begin{equation} \label{eq:ub}
     \|(\la^{y,J}_j)'\|_{L^p(I_1)} \le C(n,p,|I_1|) \, \max_{1 \le i \le n} \|a_i\|^{1/i}_{C^{n-1,1}(\overline K)}.
  \end{equation}  

  Let $J,J_0 \in \cC^y$ be arbitrary. 
  By \cite[Lemma~3.6]{ParusinskiRainerAC}, $(\la^y_j)'$ as well as $(\la^{y,J_0}_j)'$ belong to $L^p(J)$ and 
  we have  
  \begin{equation*}
    \sum_{j=1}^n \|(\la^y_j)'\|_{L^p(J)}^p 
    = \sum_{j=1}^n \|(\la^{y,J}_j)'\|_{L^p(J)}^p = \sum_{j=1}^n \|(\la^{y,J_0}_j)'\|_{L^p(J)}^p.  
  \end{equation*}
  Thus,  
  \begin{align*}
    \sum_{j=1}^n \|(\la^y_j)'\|_{L^p(\Om^y)}^p 
    &= \sum_{J \in \cC^y} \sum_{j=1}^n \|(\la^y_j)'\|_{L^p(J)}^p 
    = \sum_{J \in \cC^y} \sum_{j=1}^n \|(\la^{y,J_0}_j)'\|_{L^p(J)}^p \\ 
    &= \sum_{j=1}^n \|(\la^{y,J_0}_j)'\|_{L^p(\Om^y)}^p 
    \le \sum_{j=1}^n \|(\la^{y,J_0}_j)'\|_{L^p(I_1)}^p.
  \end{align*}
  In particular, by \eqref{eq:ub},
  \begin{align*}
    \|\p_t\la(~,y)\|_{L^p(\Om^y)} = \|(\la^y_1)'\|_{L^p(\Om^y)} 
    \le C(n,p,K) \, \max_{1 \le i \le n} \|a_i\|^{1/i}_{C^{n-1,1}(\overline K)},
  \end{align*}
  and so, by Fubini's theorem, 
  \begin{align*} 
      \int_{V \cap K} |\p_1 \la(x)|^p\, dx &= \int_{V_1 \cap L} \int_{\Om^y} |\p_1 \la(t,y)|^p\, dt\, dy 
      \\
      &\le \Big(C(n,p,K) \, \max_{1 \le i \le n} \|a_i\|^{1/i}_{C^{n-1,1}(\overline K)} \Big)^p \int_{V_1 \cap L} \, dy,
    \end{align*}  
  and thus 
  \begin{equation*}
    \|\p_1 \la \|_{L^p(V \cap K)} \le C(n,p,K) \, \max_{1 \le i \le n} \|a_i\|^{1/i}_{C^{n-1,1}(\overline K)}.   
  \end{equation*} 
  The other partial derivatives $\p_i \la$, $i \ge 2$, are treated analogously.
  This implies \eqref{multbound}, where $W$ is the (finite) union of the boxes $K$. 
\end{proof}

\begin{remark} \label{remark4}
  This can be improved slightly if $\overline V$ has just finitely many recesses:  
  in this case the constant in \eqref{multbound} depends only on $m$, $n$, $p$, and $\on{diam}(\overline V)$.
  For simplicity let us assume that $\overline V$ is convex. Then in the previous proof we need not restrict to 
  the open boxes $K$. Instead of $I_1$ we may work with the interval $\overline {V^y}$ and \eqref{eq:ub} can be 
  replaced by
  \begin{equation} \label{eq:ub2}
     \|(\la^{y,J}_j)'\|_{L^p(\overline {V^y})} \le C(n,p) \max\{1,\on{diam}(\overline V)^{1/p}\} 
     \, \max_{1 \le i \le n} \|a_i\|^{1/i}_{C^{n-1,1}(\overline V)}.
  \end{equation} 
\end{remark}

\section{Applications} \label{applications}

In this section we present three applications of our main results Theorems \ref{main} and \ref{main2}. 
First we improve upon a result due to Spagnolo \cite{Spagnolo00} on local solvability of certain systems of 
pseudo-differential equations. Secondly, we obtain a lifting theorem for differentiable mappings into orbit spaces 
of finite group representations. 
As a third application we give a sufficient condition for a multi-valued function to be of Sobolev class 
$W^{1,p}$ in the sense of Almgren.
We also want to point out that our results were used in \cite{BeckBecker-KahnHanin16}.

\subsection{Local solvability of pseudo-differential equations} \label{PDE}

In \cite{Spagnolo00} Spagnolo proved that the pseudo-differential $n \times n$ system
\begin{align} \label{pseudo}
  u_t + i A(t,D_x) u + B(t,D_x) u = f(t,x), \qquad (t,x) \in I \times U \subseteq \R \times \R^m,
\end{align}
where $A \in C^\infty(I,S^1(\R^m))^{n \times n}$, $B \in C^0(I,S^0(\R^m))^{n \times n}$ 
are matrix symbols of order $1$ and $0$, respectively, and $A(t,\xi)$ is homogeneous 
of degree $1$ in $\xi$ for $|\xi|\ge 1$,  
is locally solvable in the Gevrey class $G^s$ for
$1 \le s \le n/(n-1)$  
and semi-globally solvable in $G^s$ for $1 < s < n/(n-1)$ under the following assumptions: 
the eigenvalues of $A(t,\xi)$ admit a parameterization $\ta_1(t,\xi),\ldots,\ta_n(t,\xi)$ such that each $\ta_j(t,\xi)$
is absolutely continuous in $t$, uniformly with respect to $\xi$, i.e.,
\begin{align} \label{A1}
  \tag{$\cA_1$}  
  |\p_t \ta_j(t,\xi)| \le \mu(t,\xi) (1+|\xi|^2)^{1/2},
\quad  \text{ with } \mu(~,\xi) \text{ equi-integrable on } I,
\end{align}
and for each $\xi$ the imaginary parts of the $\ta_j(t,\xi)$ do not change sign for varying $t$ and $j$, i.e., 
\begin{align} \label{A2}
  \tag{$\cA_2$}  
  \forall \xi \quad \text{ either } \on{Im} \ta_j(t,\xi) \ge 0,\quad \forall t,j, 
  \quad \text{ or } \on{Im} \ta_j(t,\xi) \le 0,\quad \forall t,j.
\end{align}

Theorem \ref{main} implies that the assumption \eqref{A1} is always satisfied. 
Indeed, this follows by applying Theorem \ref{main} to the characteristic polynomial of the matrix 
$(1+|\xi|^2)^{-1/2}A(t,\xi)$ and noting that the entries of $(1+|\xi|^2)^{-1/2}A(t,\xi)$ and its 
iterated partial derivatives with respect to $t$ are globally bounded in $\xi$, since $A(t,\xi)$ is a symbol of order 
$1$. 

In particular, the scalar equation
\begin{equation}
  \p_t^n u + \sum_{j=1}^n a_j(t,D_x) \p_t^{n-j} u = f(t,x),
\end{equation}
where $u, f$ are scalar functions and $a_j(t,D_x)$ is a pseudo-differential operator of order $j$ 
with principal symbol $a_j^0(t,\xi)$ smooth in $t$, 
is locally solvable in $G^s$ for $1 \le s \le n/(n-1)$ and semi-globally solvable in $G^s$ for $1 < s < n/(n-1)$ 
provided that the roots $\ta_1(t,\xi),\ldots,\ta_n(t,\xi)$ of 
\[
  (iZ)^n + \sum_{j=1}^n a_j^0(t,\xi) (iZ)^{n-j} = 0
\]
satisfy assumption \eqref{A2}; cf.\ \cite[Corollary 2]{Spagnolo00}.

A crucial tool in the proof is the technique of quasi-diagonalization for a Sylvester matrix, 
introduced by \cite{Jannelli89} for 
weakly hyperbolic problems and then refined by \cite{DAnconaSpagnolo98}.

Actually, by Theorem \ref{main}, the above conclusions hold provided that the matrix symbol $A(t,\xi)$ is just of 
class $C^{n-1,1}$ in time $t$.

\begin{theorem}
    The pseudo-differential $n \times n$ system \eqref{pseudo}, 
    where $A \in C^{n-1,1}(I,S^1(\R^m))^{n \times n}$, $B \in C^0(I,S^0(\R^m))^{n \times n}$,  
    and $A(t,\xi)$ is homogeneous 
    of degree $1$ in $\xi$ for $|\xi|\ge 1$,   
    is locally solvable in the Gevrey class $G^s$ for
    $1 \le s \le n/(n-1)$  
    and semi-globally solvable in $G^s$ for $1 < s < n/(n-1)$ provided that the eigenvalues 
    $\ta_1(t,\xi),\ldots,\ta_n(t,\xi)$ of 
    $A(t,\xi)$ satisfy \eqref{A2}.
\end{theorem}  

\begin{proof}
  Theorem \ref{main} implies \eqref{A1} provided that $A(t,\xi)$ is $C^{n-1,1}$ in $t$. 
  Then the proof in \cite{Spagnolo00} yields the result. 
\end{proof}

\subsection{Lifting mappings from orbit spaces} \label{lifting}

Let $G$ be a finite group and let $\rh : G \to \on{GL}(V)$ be a representation of $G$ in a finite 
dimensional complex vector space $V$. 
By Hilbert's theorem, the algebra $\C[V]^G$ of $G$-invariant polynomials on $V$ is finitely generated. 
We consider the categorical quotient 
$V \cq G$, i.e., the affine algebraic variety with coordinate ring $\C[V]^G$, and the 
morphism $\pi : V \to V \cq G$ defined by the embedding $\C[V]^G \to \C[V]$.
Since $G$ is finite, $V\cq G$ coincides with the orbit space $V/G$.
Let $\si_1,\ldots,\si_n$ be a 
system of homogeneous generators of $\C[V]^G$ with positive degrees $d_1,\ldots,d_n$. 
Then we can identify $\pi$ with the mapping of invariants $\si=(\si_1,\dots\si_n) : V \to \si(V) \subseteq \C^n$ 
and the orbit space $V/G$ with the image $\si(V)$.

Let $U \subseteq \R^m$ be open, and $k \in \N$. 
Consider a mapping $f \in C^{k-1,1}(U,\si(V))$, i.e., $f$ is of H\"older class $C^{k-1,1}$ as mapping $U \to \C^n$ 
with the image $f(U)$ contained in $\si(V) \subseteq \C^n$. 
We say that a mapping $\overline f : U \to V$ is a \emph{lift of $f$ over $\si$} if $f = \si \o \overline f$. 
It is natural to ask how \emph{regular} a lift of $f$ can be chosen.
This question is independent of the choice of generators of $\C[V]^G$, since any two choices differ by a 
polynomial diffeomorphism. 
This and similar problems were studied in \cite{AKLM00}, \cite{KLMR05}, \cite{KLMR06}, \cite{KLMR08}, \cite{KLMRadd}, 
\cite{LMRac}, \cite{RainerRG}, \cite{ParusinskiRainer14}.

\[
\xymatrix{
   && V \ar@{->>}^{\si}[d] \ar@(dr,ur)_G & \\
   U \ar_{f}[rr] \ar@{-->}^{\overline f}[rru] &&  \si(V) \ar@{=}[d] \ar@{^{ (}->}[r] & \C^n \\
   && V/G & 
}
\]

The subject of this paper, i.e., optimal regularity of roots of polynomials, is just a special case of  
this problem: let the symmetric group $\on{S}_n$ act on $\C^n$ by permuting the coordinates. Then $\C[\C^n]^{\on{S}_n}$ 
is generated by the elementary symmetric polynomials $\si_j(z) = \sum_{i_1<\cdots<i_j} z_{i_1} \cdots z_{i_j}$, 
$\C^n/\on{S}_n = \si(\C^n) = \C^n$, and $f : U \to \si(\C^n)$ amounts to a family of complex monic polynomials $P_f$ 
with coefficients $(-1)^j f_j$, $j =1,\ldots,n$, in view of Vieta's formulas. Lifting $f$ over $\si$ precisely 
means choosing the roots of $P_f$.

As an application of our main Theorems \ref{main} and \ref{main2} we obtain the following lifting result for finite groups.
Following Noether's proof of Hilbert's theorem we associated a suitable polynomial and use the regularity result for 
its roots. In the following $Gv:=\{gv : g \in G\}$ denotes the orbit through $v$.

\begin{theorem}
  Let $\rh : G \to \on{GL}(V)$ be a complex finite dimensional representation of a finite group $G$. 
  Let $\si_1,\ldots,\si_n$ be a system of homogeneous generators of $\C[V]^G$.
  Decompose $V = \bigoplus_{i=1}^\ell V_i$ into irreducible subrepresentations of $G$, and let 
  \[
    k := \max_{i=1,\ldots,\ell} \min_{v \in V_i \setminus \{0\}} |Gv|.
  \]
  Then:
  \begin{enumerate}
    \item If $c \in C^{k-1,1}(I,\si(V))$, where $I \subseteq \R$ is a compact interval, 
    then any continuous lift $\overline c \in C^0(I,V)$ of $c$ is absolutely continuous and belongs to the Sobolev space 
    $W^{1,p}(I,V)$ for every $1 \le p <k/(k-1)$.
    If $\cC$ is a bounded subset of $C^{k-1,1}(I,\si(V))$, 
    then $\overline \cC := \{\overline c \in C^0(I,V) : \si \o \overline c \in \cC\}$ is bounded in 
    $W^{1,p}(I,V)$ for every $1 \le p <k/(k-1)$.
    \item If $f \in C^{k-1,1}(U,\si(V))$, where $U \subseteq \R^m$ is open, 
    and $\overline f \in C^0(\Om,V)$ is a continuous lift of $f$ on a relatively compact open subset $\Om \Subset U$,
    then $\overline f$ belongs to the Sobolev space 
    $W^{1,p}(\Om,V)$ for every $1 \le p <k/(k-1)$.
    If $\cF$ is a bounded subset of $C^{k-1,1}(U,\si(V))$, 
    then $\overline \cF := \{\overline f \in C^0(\Om,V) : \si \o \overline f \in \cF\}$ is bounded in 
    $W^{1,p}(\Om,V)$ for every $1 \le p <k/(k-1)$.
  \end{enumerate}
\end{theorem}

Note that there always exists a continuous lift $\overline c$ of $c \in C^0(I,\si(V))$; see \cite[Theorem 5.1]{LMRac}.

\begin{proof} 
  By treating the irreducible subrepresentations separately, we may assume without loss of generality that 
  $\rh$ is irreducible. 
  Fix a non-zero vector $v \in V$ such that $|Gv|$ is minimal.
  Choose a $G$-invariant Hermitian inner product $\<~,~\>$ on $V$,  
  and associate with $g \in G$ the linear form $\ell_g : V \to \C$ defined by $\ell_g (x):= \<x,gv\>$. 
  Choose a numbering of the left coset $G/G_v =\{g_1,\ldots,g_k\}$, where $G_v = \{g \in G: gv =v\}$ and $k = |Gv|$, 
  and set $\ell_i := \ell_{g_i}$ for $i=1,\ldots,k$. 
  Then the action of $G$ on $G/G_v$ by left multiplication induces a permutation of the set $\{g_1,\ldots,g_k\}$, and 
  thus
  \[
    a_j := (-1)^j \sum_{1 \le i_1<\cdots< i_j\le k} \ell_{i_1} \cdots \ell_{i_j}, \quad j=1,\ldots,k,
  \]  
  are $G$-invariant polynomials on $V$. So $a_j = p_j \o \si$ for polynomials $p_j \in \C[\C^n]$, and the polynomial
  $P_a \in \C[V]^G[Z]$ given by
  \[
    P_{a(x)}(Z) = Z^k + \sum_{j=1}^k a_j(x)Z^{k-j} = \prod_{j=1}^k (Z - \ell_j(x)), \quad x \in V, 
  \]
  factors through the polynomial $P_p \in \C[\C^n][Z]$, i.e., $P_a = P_{p \o \si}$.
  Applying Theorem \ref{main} to $P_{p(c(t))}$, $t \in I$, 
  we find that $t \mapsto \ell_i (\overline c(t))= \<\overline c(t),g_i v\>$,  $i=1,\ldots,k$, 
  belongs to $W^{1,p}(I)$ for each $1\le p <k/(k-1)$. 
  Since $\rh$ is irreducible, the orbit $Gv$ spans $V$ and (1) follows.  
  Analogously, (2) follows from Theorem \ref{main2}. 
\end{proof}

As a consequence one obtains a similar result for polar representations of reductive algebraic groups, since the 
lifting problem can be reduced to the action of the corresponding generalized Weyl group which is finite; cf.\ 
\cite{LMRac} or \cite{RainerRG}.

\subsection{Multi-valued Sobolev functions} 

In \cite{Almgren00} Almgren developed a theory of $n$-valued Sobolev functions 
and proved the existence of $n$-valued minimizers 
of the Dirichlet energy functional. See also \cite{De-LellisSpadaro11} for simpler proofs. 

An $n$-valued function is a mapping with values in the set $\cA_n(\R^\ell)$ of unordered $n$-tuples of points in $\R^\ell$. 
Let us denote by $[x] = [x_1,\ldots,x_n]$ the unordered $n$-tuple consisting of $x_1,\ldots,x_n \in \R^\ell$; then 
$[x_1,\ldots,x_n] = [x_{\si(1)},\ldots,x_{\si(n)}]$ for each permutation $\si \in \on{S}_n$. 
The set $\cA_n(\R^\ell) = \{[x] = [x_1,\ldots,x_n] : x_i \in \R^\ell\}$ forms a 
complete metric space when endowed with the metric
\[
  d([x],[y]) := \min_{\si \in \on{S}_n} \Big(\sum_{i=1}^n |x_i - y_{\si(i)}|^2\Big)^{1/2}.
\] 
Almgren proved that 
there is an integer $N= N(n,\ell)$, a positive constant $C= C(n,\ell)$, and an  
injective mapping $\De : \cA_n(\R^\ell) \to \R^N$ such that 
$\on{Lip}(\De) \le 1$ and $\on{Lip}(\De|^{-1}_{\De(\cA_n(\R^\ell))}) \le C$; moreover, there is a
Lipschitz retraction of $\R^N$ onto $\De(\cA_n(\R^\ell))$.

One can use this bi-Lipschitz embedding to define Sobolev spaces of $n$-valued functions: 
for open $U \subseteq \R^m$ and $1 \le p \le \infty$ define 
\[
  W^{1,p}(U,\cA_n(\R^\ell)) := \{f : U \to \cA_n(\R^\ell) : \De \o f \in W^{1,p}(U,\R^N)\}.  
\]
For an intrinsic definition see \cite[Definition~0.5 and Theorem~2.4]{De-LellisSpadaro11}.

Let us identify $\R^2 \cong \C$. 
Theorem \ref{main} implies a sufficient condition for an $n$-valued function $U \to \cA_n(\C)$ to belong to the 
Sobolev spaces $W^{1,p}(U,\cA_n(\C))$ for every $1 \le p < n/(n-1)$; see Theorem \ref{multivalued} below.

We shall use the following terminology. By a \emph{parameterization} of an $n$-valued function $f : U \to \cA_n(\C)$ we mean a
function $\vh: U \to \C^n$ such that $f(x) = [\vh(x)] = [\vh_1(x),\ldots,\vh_n(x)]$ for all $x \in U$. 
Let $\pi : \C^n \to \cA_n(\C)$ be defined by $\pi(z) := [z]$; it is a Lipschitz mapping with $\on{Lip}(\pi) = 1$.
Then a parameterization of $f$ amounts to a lift $\vh$ of $f$ over $\pi$, i.e.,
$f = \pi \o \vh$.  
The elementary symmetric polynomials induce a bijective mapping  
$a : \cA_n(\C) \to \C^n$, 
\[
  a_j ([z_1,\ldots,z_n]) := (-1)^j \sum_{i_1<\cdots<i_j} z_{i_1} \cdots z_{i_j}, \quad  1 \le j \le n.
\]
In other words, monic complex polynomials of degree $n$ are in one-to-one correspondence 
with their unordered $n$-tuples of roots.

\begin{theorem} \label{multivalued}
  Let $U \subseteq \R^m$ be open and
  let $f : U \to \cA_n(\C)$ be continuous.
  If $a \o f \in C^{n-1,1}(U,\C^n)$, then $f \in W^{1,p}(V,\cA_n(\C))$ for each relatively compact open $V \Subset U$ 
  and each $1 \le p < n/(n-1)$. 
  Moreover,
  \begin{equation*}
    \|\nabla (\De \o f)\|_{L^p(V)} 
    \le C(m,n,p,\cK,\De) \big(1 + \max_{1 \le j \le n} \|a_j \o f\|^{1/j}_{C^{n-1,1}(\overline W)}\big),
  \end{equation*}
  where $\cK$ is any finite cover of $\overline V$ by open boxes $\prod_{i=1}^m (\al_i,\be_i)$ 
  contained in $U$ and $W = \bigcup \cK$.
\end{theorem}

\begin{proof}
  Fix $V \Subset U$.
  We must show that $\De \o f$ is an element of $W^{1,p}(V ,\R^N)$.
  Clearly, $\De \o f : U \to \R^N$ is continuous. 
  The set $V$ can be covered by finitely many open boxes $K = \prod_{i=1}^m I_i$ contained in U.
  Let $e_i$ be the $i$th standard unit vector in $\R^m$. 
  Denote by $K_i$ the orthogonal projection of $K$ onto the hyperplane $e_i^\bot$. 
  For each $y \in K_i$ we have $I_i = \{t \in \R : y+te_i \in K\}$.

  By Theorem \ref{main}, 
  $I_i \ni t \mapsto f(y+te_i)$ admits an absolutely continuous parameterization $\vh_{i,y}$ such that, 
  for $1 \le p < n/(n-1)$, 
  \begin{align*}  
   \| \vh_{i,y}' \|_{L^p(I_i)}  
   &\le C(n,p)\max\{1,|I_i|^{1/p}\}   
   \max_{1 \le j \le n} \|a_j \o f\|^{1/j}_{C^{n-1,1}(\overline K)}.
  \end{align*}  
  Thus, 
  $I_i \ni t \mapsto \De (f(y+te_i)) = \De (\pi(\vh_{i,y}(t)))$ is absolutely continuous and 
  \begin{align*}  
   \| (\De \o \pi \o \vh_{i,y})' \|_{L^p(I_i)}  
   &\le C(m,n,p,|I_i|, \De) \big(1 +  
   \max_{1 \le j \le n} \|a_j \o f\|^{1/j}_{C^{n-1,1}(\overline K)}\big),
  \end{align*}
  since composition with the Lipschitz mapping $\De \o \pi$ maps $W^{1,p}(I_i,\C^n)$ to $W^{1,p}(I_i,\R^N)$
  in a bounded way; see \cite[Theorem 1]{MarcusMizel79}.  
  By Fubini's theorem,
  \begin{align*}
    \int_K |\p_i (\De \o f)|^p \, dx =  \int_{K_i} \int_{I_i} |(\De \o \pi \o \vh_{i,y})'|^p \, dt \, dy, 
  \end{align*}
  and the statement follows.
\end{proof}

\[
\xymatrix{
   &&& \C^n \ar@{->>}^{\pi}[d] \ar[drr] && \\
   I_i \ar@{^{ (}->}[r] \ar@<1ex>^{\vh_{i,y}}[urrr]
   &K \ar^{f}[rr] \ar[drr]  &&  \cA_n(\C) \ar@<-2pt>_{a}[d] \ar@{^{ (}->}^{\De}[rr] && \R^N \\
   &&& \C^n \ar@<-2pt>_{a^{-1}}[u] && 
}
\]

In particular, the roots of a polynomial $P_a$ of degree $n$ with coefficients 
$a_j \in C^{n-1,1}(U)$, $j = 1,\ldots,n$, form an $n$-valued function  
$\la : U \to \cA_n(\C)$ which belongs to $W^{1,p}_{\on{loc}}(U,\cA_n(\C))$ for each $1 \le p < n/(n-1)$; 
in fact, it is well-known that $\la : U \to \cA_n(\C)$ is 
continuous (cf.\ \cite{Kato76} or \cite[Theorem 1.3.1]{RS02}).
Theorem \ref{multivalued} implies that the push-forward  
\begin{equation*}
  (a^{-1})_* : C^{n-1,1}(U,\C^n) \to \bigcap_{1 \le p <n/(n-1)} W^{1,p}_{\on{loc}}(U,\cA_n(\C)).
\end{equation*}
is a bounded mapping.

We remark that much more is true in the case of \emph{real} $n$-valued functions. 
In this situation the elementary symmetric polynomials induce a bijective mapping $a : \cA_n(\R) \to H_n$, 
where $H_n$ is a closed semialgebraic subset of $\R^n$, namely, the space of \emph{hyperbolic} polynomials 
of degree $n$ (i.e., polynomials with all roots real). Then the mapping
\begin{equation*}
  (a^{-1})_*  : C^{n-1,1}(U,H_n) \to C^{0,1}(U,\cA_n(\R)),
\end{equation*}
is bounded.
It is easy to see that the projection $\pi : \R^n \to \cA_n(\R)$ admits a continuous section $\th$, for instance, by 
ordering the components increasingly.
Then
we have a bounded mapping   
\begin{equation*}
  (\th \o a^{-1})_* : C^{n-1,1}(U,H_n) \to C^{0,1}(U,\R^n).
\end{equation*}
All this essentially follows from Bronshtein's theorem \cite{Bronshtein79}; see \cite{ParusinskiRainerHyp}.

\begin{remark}
  Let $\Ph : \cA_n(\C) \to \cA_n(\R)$ be a Lipschitz function.
  If $f \in  W^{1,p}(U,\cA_n(\C))$, then $\Ph \o f \in W^{1,p}(U,\cA_n(\R))$ and it admits a parameterization 
  $\th \o \Ph \o f \in W^{1,p}(U,\R^n)$.     
  This follows (again by \cite[Theorem 1]{MarcusMizel79}) 
  from the following diagram in which all vertical arrows are Lipschitz: the arrows in the lower 
  row by Almgren's results, 
  and $\th$ is Lipschitz, since $d([x],[y]) = |\th([x]) - \th([y])|$ for $[x],[y] \in \cA_n(\R)$.
  \[
    \xymatrix{
    &&&& \R^n \ar@<-3pt>@{->>}_{\pi}[d] \\
    U \ar@<1ex>[rrrru] \ar^{f}[rr] && \cA_n(\C) \ar@<-3pt>@{_{ (}->}_{\De_2}[d] \ar^{\Ph}[rr] 
    && \cA_n(\R) \ar@<-3pt>_{\th}[u] \ar@<-3pt>@{_{ (}->}_{\De_1}[d]  \\
    && \R^{N_2} \ar@<-3pt>[u] \ar[rr] && \R^{N_1} \ar@<-3pt>[u] 
    }
  \]
  Every Lipschitz function $\ph : \C \to \R$ induces a Lipschitz functions $\Ph : \cA_n(\C) \to \cA_n(\R)$ by setting 
  $\Ph([z]) := [\ph(z_1),\ldots,\ph(z_n)]$. 
  In particular, we can take $\vh(z) = |z|$, $\vh(z) = \Re (z)$, or $\vh(z) = \Im(z)$. 
  In view of Theorem \ref{multivalued}, we may conclude that the real and imaginary parts of the roots of 
  a monic polynomial $P_a$ of degree $n$ with coefficients in $C^{n-1,1}(U)$ admit continuous parameterization 
  that are of class $W^{1,p}_{\on{loc}}(U,\R^n)$ for each $1 \le p < n/(n-1)$. The same holds for the absolute values.
  But note that real and imaginary parts of the roots do not allow continuous parameterizations simultaneously!     
\end{remark}

\appendix

\section{Illustration of the proof in simple cases} \label{appendix}

Let us illustrate the proof of Theorem \ref{main} for polynomials $P_a$ of degree 3 and 4. 
For simplicity we assume that $P_a$ is in Tschirnhausen form.

\subsection*{Degree 3}

In degree 3 
Proposition \ref{induction} is trivial: the factors of a splitting are at most of degree 2; 
so \eqref{muab} reduces to $\|\mu'\|_{L^p(I)} = \|(\tilde b_2^{1/2})' \|_{L^p(I)}$ if $n_b =2$ and   
$\mu \equiv 0$ if $n_b=1$.  

Let $(\al,\be) \subseteq \R$ be a bounded open interval.
Let 
\[
	P_{\tilde a(t)}(Z) = Z^3 + \tilde a_2(t) Z + \tilde a_3(t), \quad   t \in (\al,\be), 
\]
be a monic polynomial of degree 3
in Tschirnhausen form with coefficients $\tilde a_2, \tilde a_3 \in C^{2,1}([\al,\be])$.
We may use Lemma \ref{lem:whitney} to extend $\tilde a_2, \tilde a_3$ to functions in $C^{2,1}([\hat \al,\hat \be])$, 
where $\hat \al = \al -1$ and $\hat \be = \be +1$, such that 
\begin{itemize}
	\item \eqref{whitney} holds for $n=3$, and
	\item for $t_0 \in (\al,\be)$ and $k \in \{2,3\}$ satisfying 
\begin{equation}
	|\tilde a_k(t_0)|^{1/k} = \max_{j=2,3} |\tilde a_j(t_0)|^{1/j} \ne 0, 
\end{equation}
and a constant $B$ satisfying \eqref{eq:constB} for $n=3$, 
there is an open interval $I \subseteq (\hat \al,\hat \be)$ containing $t_0$ such that 
\begin{equation} \label{d3:keya}
	M |I| + \|(\tilde a_2^{1/2})'\|_{L^1(I)} + \|(\tilde a_3^{1/3})'\|_{L^1(I)} = B |\tilde a_k(t_0)|^{1/k},
\end{equation}
where
\begin{equation} \label{d3:M}
	M = \max_{j=2,3} (\Lip_I(\tilde a_j^{(2)}))^{1/3} |\tilde a_k(t_0)|^{(3-j)/(3k)}.
\end{equation}
\end{itemize}
We have a splitting $P_{\tilde a}(t) = P_{b}(t) P_{b^*}(t)$, $t \in I$ (see Lemma \ref{lem:assThm1implyassProp3}).

\subsubsection*{Case $n_b = 2$}
In this case 
\[
	P_{b(t)}(Z) = Z^2 + b_1(t) Z + b_2(t), \quad t \in I,  
\] 
and after Tschirnhausen transformation 
\[
	P_{\tilde b(t)}(Z) = Z^2 +\tilde b_2(t), \quad t \in I.  
\]
The coefficients $b_1$, $b_2$, and $\tilde b_2$ are given by \eqref{eq:b_i0} and \eqref{eq:b_i} for $n_b =2$. 
They are of class $C^{2,1}(\overline I)$ since $\tilde a_k$ does not vanish on $I$ (by \eqref{eq:ass11}). 
If $\mu \in C^0(I)$ is a continuous root of $P_{\tilde b}$, then Lemma \ref{lem:B} and Lemma \ref{lem:Lpbak} imply 
\begin{equation} \label{d3:mu}
	\|\mu'\|^*_{L^p(I)} = \|(\tilde b_2^{1/2})' \|^*_{L^p(I)} \le C(p) |I|^{-1} |\tilde a_k(t_0)|^{1/k}, \quad  1\le p < 2.
\end{equation}
Moreover, by \eqref{b1prime},
\begin{equation} \label{d3:b1}
	\|b_1'\|^*_{L^p(I)} \le  C |I|^{-1} |\tilde a_k(t_0)|^{1/k}.
\end{equation}

\subsubsection*{Case $n_b = 1$}
In this case $P_{b(t)}(Z) = Z + b_1(t)$, $P_{\tilde b(t)}(Z) = Z$, and $\mu \equiv 0$.
In particular, \eqref{d3:mu} and \eqref{d3:b1} are still valid.

Let $\la \in C^0((\al,\be))$ be a continuous root of $P_{\tilde a}$. 
We extend $\la$ continuously to $(\hat \al,\hat \be)$ such that $\la$ is a root of $P_{\tilde a}$ on $(\hat \al,\hat \be)$. 
Assume that, on $I$, $\la$ is a root of $P_b$; then
\[
	\la(t) = - \frac{b_1(t)}{n_b} + \mu(t), \quad t \in I.	
\] 
By \eqref{d3:mu}, \eqref{d3:b1}, and \eqref{inclusions0},
\begin{align*} 
	\|\la'\|^*_{L^p(I)} 
	&\le C(p) |I|^{-1} |\tilde a_k(t_0)|^{1/k}
	\\ \notag
	&= C(p)B^{-1} \Big( M  + \|(\tilde a_2^{1/2})'\|^*_{L^1(I)} + \|(\tilde a_3^{1/3})'\|^*_{L^1(I)} \Big)
	\\ \notag
	&\le C(p)B^{-1} \Big( \hat A  + \|(\tilde a_2^{1/2})'\|^*_{L^p(I)} + \|(\tilde a_3^{1/3})'\|^*_{L^p(I)} \Big), 
\end{align*}
where $\hat A := \max_{j=2,3} \|\tilde a_j\|^{1/j}_{C^{2,1}([\hat \al,\hat \be])}$ which dominates $M$ as defined in 
\eqref{d3:M} (see the proof of Lemma \ref{lem:local}).
By Proposition \ref{cover} 
(applied to $\tilde a_j$ instead of $\tilde b_i$ and \eqref{d3:keya} instead of \eqref{assumption2aG}) 
and Lemma \ref{lem:extend}, we may conclude that 
$\la$ is absolutely continuous on $(\hat \al,\hat \be)$ and satisfies
\[
	\|\la'\|_{L^p((\hat \al,\hat \be))} \le C(p) \Big( \hat A (\hat \be - \hat \al)^{1/p} 
	+ \|(\tilde a_2^{1/2})'\|_{L^p((\hat \al,\hat \be))} + \|(\tilde a_3^{1/3})'\|_{L^p((\hat \al,\hat \be))} \Big);
\]
the constant $B$ is universal.
Using \eqref{est} and \eqref{whitney}, we find 
\[
	\|\la'\|_{L^p((\al,\be))} 
	\le C(p) \max\{1,(\be -\al)^{1/p}\} \max_{j=2,3} \|\tilde a_j\|^{1/j}_{C^{2,1}([\al,\be])}, \quad 1 \le p < 3/2. 
\]

\subsection*{Degree 4}

In degree 4 the interesting case is when after splitting one of the factors has degree 3. 
Then the conclusion of Proposition \ref{induction} is obtained by a second splitting which further reduces the degree.

Let $(\al,\be) \subseteq \R$ be a bounded open interval.
Let 
\[
	P_{\tilde a(t)}(Z) = Z^4 + \tilde a_2(t) Z^2 + \tilde a_3(t) Z + \tilde a_4(t), \quad   t \in (\al,\be), 
\]
be a monic polynomial of degree 4
in Tschirnhausen form with coefficients $\tilde a_2, \tilde a_3, \tilde a_4 \in C^{3,1}([\al,\be])$. 
As in degree 3 
we may assume that $\tilde a_2, \tilde a_3, \tilde a_4$ are functions in $C^{3,1}([\hat \al,\hat \be])$
(where $\hat \al = \al -1$ and $\hat \be = \be +1$)
such that 
\begin{itemize}
	\item \eqref{whitney} holds for $n=4$, and
	\item for $t_0 \in (\al,\be)$ and $k \in \{2,3,4\}$ satisfying 
\begin{equation}
	|\tilde a_k(t_0)|^{1/k} = \max_{j=2,3,4} |\tilde a_j(t_0)|^{1/j} \ne 0, 
\end{equation}
and a constant $B$ satisfying \eqref{eq:constB} for $n=4$, 
there is an open interval $I \subseteq (\hat \al,\hat \be)$ containing $t_0$ such that 
\begin{equation} \label{d4:keya}
	M |I| + \sum_{j=2}^4\|(\tilde a_j^{1/j})'\|_{L^1(I)}  = B |\tilde a_k(t_0)|^{1/k},
\end{equation}
where
\begin{equation} \label{d4:M}
	M = \max_{j=2,3,4} (\Lip_I(\tilde a_j^{(3)}))^{1/4} |\tilde a_k(t_0)|^{(4-j)/(4k)}.
\end{equation}
\end{itemize}
We have a splitting $P_{\tilde a}(t) = P_{b}(t) P_{b^*}(t)$, $t \in I$.

\subsubsection*{Case $n_b=3$} In this case
\[
	P_{b(t)}(Z) = Z^3 + b_1(t) Z^2 + b_2(t) Z + b_3(t), \quad t \in I,  
\] 
and after Tschirnhausen transformation 
\[
	P_{\tilde b(t)}(Z) = Z^3 +\tilde b_2(t)Z + \tilde b_3(t), \quad t \in I.  
\]
The coefficients $b_1$, $b_2$, $b_3$ and $\tilde b_2$, $\tilde b_3$ are given by \eqref{eq:b_i0} and \eqref{eq:b_i} 
for $n_b =3$. 
They are of class $C^{3,1}(\overline I)$ since $\tilde a_k$ does not vanish on $I$. 

In this situation we have to work harder to obtain the conclusion of Proposition \ref{induction}: we must split again.
Let $I' := I \setminus \{t \in I : \tilde b_2 (t) = \tilde b_3 (t) = 0\}$. 
For each $t_1 \in I'$ choose $\ell \in \{2,3\}$ such that 
\[
	|\tilde b_\ell(t_0)|^{1/\ell} = \max_{j=2,3} |\tilde b_i(t_0)|^{1/i} \ne 0.
\]
There is an open interval $J=J(t_1)$, $t_1 \in J \subseteq I'$, such that 
\begin{equation} \label{d4:keyb}
	| J|  |I|^{-1}  {|\tilde a_k(t_0)|^{1/k}}  
  + \|(\tilde b _2^{1/2})'\|_{L^1 (J)} + \|(\tilde b _3^{1/3})'\|_{L^1 (J)} =  D |\tilde b_\ell(t_1)|^{1/\ell},
\end{equation}
for a constant $D$ satisfying \eqref{D} for $n_b=3$. 
Then we have a splitting $P_{\tilde b}(t) = P_{c}(t) P_{c^*}(t)$, $t \in J$; see Section \ref{subintervals} and p.~\pageref{D}. 

Let $\mu \in C^0(I)$ be a continuous root of $P_{\tilde b}$. 
We may assume that 
\[
	\tilde \mu (t) := \mu(t) + \frac{c_1(t)}{n_c}, \quad  t \in J ,
\]
is a root of $P_{\tilde c}$ in $J$. We have $n_c \le 2$.
If $n_c =2$, then, 
in analogy to \eqref{d3:mu} and \eqref{d3:b1},
\begin{equation} \label{d4:timu}
	\|\tilde \mu'\|^*_{L^p(J)} = \|(\tilde c_2^{1/2})' \|^*_{L^p(J)} \le C(p) |J|^{-1} |\tilde b_\ell(t_1)|^{1/\ell}, 
	\quad  1\le p < 2.
\end{equation}
and
\begin{equation} \label{d4:c1} 
	\|c_1'\|^*_{L^p(J)} \le  C |J|^{-1} |\tilde b_\ell(t_1)|^{1/\ell}.
\end{equation}
In the case that $n_c=1$ we have $P_{c(t)}(Z) = Z + c_1(t)$, $P_{\tilde c(t)}(Z) = Z$, and $\tilde \mu \equiv 0$. 
In particular, \eqref{d4:timu} and \eqref{d4:c1} are still valid.

Thus, \eqref{d4:keyb}, \eqref{d4:timu}, \eqref{d4:c1}, and \eqref{inclusions0} imply 
\begin{align*}
	\|\mu'\|^*_{L^p(J)} 
	&\le C(p) |J|^{-1} |\tilde b_\ell(t_1)|^{1/\ell}
	\\
	&= C(p)D^{-1}  
	\Big(  |I|^{-1}  {|\tilde a_k(t_0)|^{1/k}}  
  + \|(\tilde b _2^{1/2})'\|^*_{L^1 (J)} + \|(\tilde b _3^{1/3})'\|^*_{L^1 (J)}\Big)
  \\
	&\le C(p)D^{-1}  
	\Big(  |I|^{-1}  {|\tilde a_k(t_0)|^{1/k}}  
  + \|(\tilde b _2^{1/2})'\|^*_{L^p (J)} + \|(\tilde b _3^{1/3})'\|^*_{L^p (J)}\Big).
\end{align*}
Using Proposition \ref{cover} to extract a countable subcollection of $\{J(t_1)\}_{t_1 \in I'}$, 
$\si$-additivity of $\|\cdot\|^p_{L^p}$ to glue the $L^p$-estimates, and Lemma \ref{lem:extend} to extend the 
estimate to $I$, we obtain
\begin{equation*} 
	\|\mu'\|_{L^p(I)} \le C(p) \Big( \| |I|^{-1}  {|\tilde a_k(t_0)|^{1/k}} \|_{L^p(I)} 
  + \|(\tilde b _2^{1/2})'\|_{L^p (I)} + \|(\tilde b _3^{1/3})'\|_{L^p (I)}\Big),
\end{equation*}
that is the conclusion of Proposition \ref{induction} (the constant $D$ is universal).
With Lemma \ref{lem:B} and Lemma \ref{lem:Lpbak} we may conclude 
\begin{equation} \label{d4:mu}
	\|\mu'\|^*_{L^p(I)} \le C(p) |I|^{-1}  {|\tilde a_k(t_0)|^{1/k}}, \quad 1 \le p < 3/2. 
\end{equation}

\subsubsection*{Case $n_b\le 2$} In this case \eqref{d4:mu} follows from \eqref{d3:mu} and \eqref{d3:b1}.

Let $\la \in C^0((\al,\be))$ be a continuous root of $P_{\tilde a}$. 
We extend $\la$ continuously to $(\hat \al,\hat \be)$ such that $\la$ is a root of $P_{\tilde a}$ on $(\hat \al,\hat \be)$. 
Assume that, on $I$, $\la$ is a root of $P_b$; then
\[
	\la(t) = - \frac{b_1(t)}{n_b} + \mu(t), \quad t \in I.	
\]
By \eqref{d4:mu}, \eqref{d3:b1}, \eqref{d4:keya}, and \eqref{inclusions0},
\begin{align*} 
	\|\la'\|^*_{L^p(I)} 
	&\le C(p) |I|^{-1} |\tilde a_k(t_0)|^{1/k}
	\\ \notag
	&= C(p)B^{-1} \Big( M  + \sum_{j=2}^4\|(\tilde a_j^{1/j})'\|^*_{L^1(I)}\Big)
	\\ \notag
	&\le C(p)B^{-1} \Big( \hat A  +  \sum_{j=2}^4\|(\tilde a_j^{1/j})'\|^*_{L^p(I)}\Big), 
\end{align*}
where $\hat A := \max_{j=2,3,4} \|\tilde a_j\|^{1/j}_{C^{3,1}([\hat \al,\hat \be])}$ dominates $M$ as defined in \eqref{d4:M}.
As in the end of the proof for degree 3, we may use
Proposition \ref{cover} and Lemma \ref{lem:extend} to glue the $L^p$-estimates, and 
\eqref{est} and \eqref{whitney} to conclude 
\[
	\|\la'\|_{L^p((\al,\be))} 
	\le C(p) \max\{1,(\be -\al)^{1/p}\} \max_{j=2,3,4} \|\tilde a_j\|^{1/j}_{C^{3,1}([\al,\be])}, \quad 1 \le p < 4/3. 
\]


\def\cprime{$'$}
\providecommand{\bysame}{\leavevmode\hbox to3em{\hrulefill}\thinspace}
\providecommand{\MR}{\relax\ifhmode\unskip\space\fi MR }
\providecommand{\MRhref}[2]{%
  \href{http://www.ams.org/mathscinet-getitem?mr=#1}{#2}
}
\providecommand{\href}[2]{#2}


\end{document}